\documentclass{article}
\usepackage{hyperref}
\usepackage{amsmath, amsthm, amssymb, amsfonts, dsfont, fancyhdr, graphicx, color, tabularx, mathtools}

\usepackage{geometry}
\usepackage[shortlabels]{enumitem} 
\usepackage[all]{xy} 

\newcommand{\CC}{{\mathbb C}}
\newcommand{\QQ}{{\mathbb Q}}

\newcommand{\ZZ}{{\mathbb Z}}

\newcommand{\PP}{\mathbb P}
\newcommand{\TT}{\mathbb T}

\newcommand{\A}{{\mathcal A}}
\newcommand{\X}{\mathcal X}
\newcommand{\mc}[1]{\mathcal{#1}}
\newcommand{\mf}[1]{\mathfrak{#1}}

\newtheorem{prop}{Proposition}[section] 
\newtheorem{cor}[prop]{Corollary} 

\newtheorem{conj}[prop]{Conjecture}
\newtheorem{thm}[prop]{Theorem}
\newtheorem{manualcorinner}{Corollary}
\newenvironment{manualcor}[1]{%
  \renewcommand\themanualcorinner{#1}%
  \manualcorinner
}{\endmanualcorinner}

\theoremstyle{definition}
\newtheorem{defn}[prop]{Definition}
\newtheorem{rmk}[prop]{Remark}
\newtheorem{ex}[prop]{Example}

\numberwithin{equation}{section} 

\title{Combinatorics of $\mathcal{X}$-variables in Finite Type Cluster Algebras}

\author{Melissa Sherman-Bennett}

\usepackage[backend=bibtex, maxbibnames=99]{biblatex}
\begin{document}

\maketitle

\begin{abstract}
We compute the number of $\mathcal{X}$-variables (also called coefficients) of a cluster algebra of finite type when the underlying semifield is the universal semifield. For classical types, these numbers arise from a bijection between coefficients and quadrilaterals (with a choice of diagonal) appearing in triangulations of certain marked surfaces. We conjecture that similar results hold for cluster algebras from arbitrary marked surfaces, and obtain corollaries regarding the structure of finite type cluster algebras of geometric type.
\end{abstract}

\section{Introduction}

Cluster algebras were introduced by Fomin and Zelevinsky in the early 2000s \cite{FZ1}, with the intent of establishing a general algebraic structure for studying dual canonical bases of semisimple groups and total positivity. A cluster algebra, or equivalently its seed pattern, is determined by an initial set of cluster variables (which we call $\A$-variables) and coefficients (which we call $\X$-variables), along with some additional data. As the terminology suggests, in the original definitions, $\A$-variables were the main focus. This is reflected in much of the research on cluster algebras to date, which focuses largely on $\A$-variables and their dynamics. However, $\X$-variables are important in total positivity, and $\X$-variables over the universal semifield have recently appeared in the context of scattering amplitudes in $\mc{N}=4$ Super Yang-Mills theory \cite{physics}. Moreover, in the setting of cluster varieties, introduced by Fock and Goncharov \cite{FG}, the $\A$- and $\X$-varieties (associated with $\A$- and $\X$-variables, respectively) are on equal footing. Fock and Goncharov conjectured that a duality holds between the two varieties \cite[Conjecture 4.3]{FG}, which was later shown to be true under fairly general assumptions \cite{GHKK}. This duality suggests that studying $\X$-variables could be fruitful both in its own right and in furthering our understanding of cluster algebras. 

A study along these lines was undertaken by Speyer and Thomas in the case of acyclic cluster algebras with principal coefficients (that is, the semifield $\PP$ is the tropical semifield $\PP=$Trop$(t_1, \dots, t_k)$ and the $\X$-variables of the initial cluster are $(t_1, \dots, t_k)$) \cite{ST} . Using methods from quiver representation theory, they found that the $\X$-variables are in bijection with roots of an associated root system and give a combinatorial description of which roots can appear in the same $\X$-cluster. Seven found that in this context, mutation of $\X$-seeds roughly corresponds to reflection across hyperplanes orthogonal to roots \cite{Sev}. However, the above results do not address $\X$-variables over the universal semifield, and the numerology in the case of principal coefficients is quite different from what we obtain here. Our proofs are also completely combinatorial.

We investigate the combinatorics of $\X$-variables for seed patterns of finite type, particularly in the case when the underlying semifield is the universal semifield. The combinatorics of $\A$-variables for finite type seed patterns is particularly rich, with connections to finite root systems \cite{FZ2} and triangulations of certain marked surfaces \cite[Chapter 5]{book}. Parker \cite{parker} conjectures, and Scherlis \cite{scherlis} gives a partial proof, that in type $A$, $\X$-variables over the universal semifield are in bijection with the quadrilaterals of these triangulations. We generalize and prove this statement for all classical types ($ABCD$).

\begin{thm} \label{quadbijection} Let $\mc{S}$ be an $\X$-seed pattern of classical type $Z_n$ over the universal semifield such that one $\X$-cluster consists of algebraically independent elements. Let $\textbf{P}$ be the marked polygon associated to type $Z_n$. Then the $\X$-variables of $\mc{S}$ are in bijection with the quadrilaterals (with a choice of diagonal) appearing in triangulations of $\textbf{P}$.
\end{thm}

We also obtain the following corollary, which follows from Theorem \ref{quadbijection} in classical types and was verified by computer in exceptional types. 

\begin{cor}\label{expairbijection} Let $\mc{R}$ be a finite type $\A$-seed pattern. There is a bijection between ordered pairs of exchangeable $\A$-variables in $\mc{R}$ and $\X(\mc{S}_{sf})$.
\end{cor}

As another corollary, we compute the number of $\X$-variables in a classical type $\X$-seed pattern $\mc{S}_{sf}$ over the universal semifield. We also compute the number $|\X(\mc{S}_{sf})|$ of $\X$-variables over the universal semifield for exceptional types using a computer algebra system. Note that all other $\X$-seed patterns with the same exchange matrices have at most as many $\X$-variables as $\mc{S}_{sf}$. The numbers $|\X(\mc{S}_{sf})|$ are listed in the second row of the following table (the numbers for $A_n$ for $n \leq6$, $D_4$, and $E_6$ were also computed in \cite{parker}). For comparison, the third row gives the number of $\X$-variables in a finite type $\X$-seed pattern $\mc{S}_{pc}$ with principal coefficients, a corollary of the results in \cite{ST}.

\begin{center}
\setlength\tabcolsep{3pt}
\begin{tabular}{| c | c | c| c | c | c| c| c| c| }
 \hline
   Type & $A_n$ & $B_n, C_n$ & $D_n$ & $E_6$ & $E_7$ & $E_8$ & $F_4$ & $G_2$ \\ \hline
$|\X(\mc{S}_{sf})|$ & $2\binom{n+3}{4}$ & $\frac{1}{3}n(n+1)(n^2+2)$ & $\frac{1}{3}n(n-1)(n^2+4n-6)$ & 770 & 2100 & 6240 & 196 & 16 \\ \hline
$|\X(\mc{S}_{pc})|$ & $n(n+1)$ & $2n^2$ & $2n(n-1)$ & $72$ & $126$ & $240$ & $48$ & $12$\\
\hline
\end{tabular}
\end{center}


\section{Seed Patterns}

We largely follow the conventions of \cite{FZ4}. 

\subsection{Seeds and Mutation}

We begin by fixing a \emph{semifield} $(\PP, \cdot, \oplus)$, a multiplicative abelian group $(\PP, \cdot)$ equipped with an (auxiliary) addition $\oplus$, a binary operation which is associative, commutative, and distributive with respect to multiplication.  

\begin{ex} Let $t_1, \dots, t_k$ be algebraically independent over $\QQ$. The \emph{universal semifield} $\QQ_{sf}(t_1, \dots, t_k)$ is the set of all rational functions in $t_1, \dots, t_k$ that can be written as subtraction-free expressions in $t_1, \dots, t_k$. This is a semifield with respect to the usual multiplication and addition of rational expressions. Note that any (subtraction-free) identity in $\QQ_{sf}(t_1, \dots, t_k)$ holds in an arbitrary semifield for any elements $u_1,  \dots, u_k$ \cite[Lemma 2.1.6]{BFZ96}.
\end{ex}

\begin{ex}The \emph{tropical semifield} Trop$(t_1, \dots, t_k)$ is the free multiplicative group generated by $t_1, \dots, t_k$, with auxiliary addition defined by 

\[\prod_{i=1}^k t_i^{a_i} \oplus \prod_{i=1}^k t_i^{b_i}= \prod_{i=1}^k t_i^{\min(a_i, b_i)}.
\]
\end{ex} 

Let $\QQ\PP$ denote the field of fractions of the group ring $\ZZ\PP$. We fix an \emph{ambient field} $\mc{F}$, isomorphic to $\QQ\PP(t_1, \dots, t_n)$.

\begin{defn} A \emph{labeled $\X$-seed} in $\PP$ is a pair $(\textbf{x}, B)$ where $\textbf{x}=(x_1, \dots, x_n)$ is a tuple of elements in $\PP$ and $B=(b_{ij})$ is a \emph{skew-symmetrizable} $n \times n$ integer matrix, that is there exists a diagonal integer matrix $D$ with positive diagonal entries such that $DB$ is skew-symmetric. 

A \emph{labeled $\A$-seed} in $\mc{F}$ is a triple $(\textbf{a}, \textbf{x}, B)$ where $(\textbf{x}, B)$ is a labeled $\X$-seed in $\PP$ and $\textbf{a}=(a_1, \dots, a_n)$ is a tuple of elements of $\mc{F}$ which are algebraically independent over $\QQ\PP$ and generate $\mc{F}$. We call $\textbf{x}$ the (labeled) $\X$\emph{-cluster}, $\textbf{a}$ the (labeled) $\A$\emph{-cluster}, and $B$ the \emph{exchange matrix} of the labeled seed $(\textbf{a}, \textbf{x}, B)$.
\end{defn}

The elements of an $\X$- (respectively $\A$-)cluster are called $\X$- (respectively $\A$-)variables. In the language of Fomin and Zelevinsky, the $\X$-cluster is the coefficient tuple, the $\A$-cluster is the cluster, and the $\X$- and $\A$-variables are coefficients and cluster variables, respectively. The notation here is chosen to parallel Fock and Goncharov's $\A$- and $\X$- cluster varieties. Note that an $\X$-seed consists only of an exchange matrix and an $\X$-cluster, but an $\A$-seed consists of an exchange matrix, an $\A$-cluster \emph{and} an $\X$-cluster. For simplicity, we use ``cluster", ``seed", etc.  without a prefix when a statement holds regardless of prefix. 

One moves from labeled seed to labeled seed by a process called \emph{mutation}.




\begin{defn}[\protect{\cite[Definition 2.4]{FZ4}}] Let $(\textbf{a},\textbf{x}, B)$ be a labeled $\A$-seed in $\mc{F}$. The \emph{$\A$-seed mutation} in direction $k$, denoted $\mu_k$, takes $(\textbf{a},\textbf{x}, B)$ to the labeled $\A$-seed $(\textbf{a}',\textbf{x}', B')$ where 

\begin{itemize}
\item The entries $b_{ij}'$ of $B'$ are given by
\begin{equation} \label{matrixmut}
	b'_{ij}=
		\begin{cases}
		-b_{ij} & \text{if } i=k \text{ or } j=k\\
		b_{ij}+b_{ik}|b_{kj}| & \text{if } b_{ik}b_{kj}>0\\
		b_{ij} & \text{else}.
		\end{cases} 
\end{equation}
\item
The $\A$-cluster $\textbf{a}'=(a_1', \dots, a_n')$ is obtained from $\textbf{a}$ by replacing the $k^\text{th}$ entry $a_k$ with an element $a_k' \in \mc{F}$ satisfying the \emph{exchange relation}

\begin{equation}\label{exchangerel}
a_k' a_k=\frac{x_k\prod\limits_{b_{ik}>0}a_i^{b_{ik}}+\prod\limits_{b_{ik}<0}a_i^{-b_{ik}}}{x_k \oplus 1}.
\end{equation}
\item The $\X$-cluster $\textbf{x}'=(x_1', \dots, x_n')$ is given by
\begin{equation}\label{xmut}
	x'_j=
		\begin{cases}
		x_j^{-1} & \text{if } j=k\\
		x_j (x_k^{sgn(-b_{kj})} \oplus 1)^{-b_{kj}} & \text{else}
		\end{cases}
\end{equation}
where $sgn(x)=0$ for $x=0$ and $sgn(x)=|x|/x$ otherwise.
\end{itemize}

Similarly, the \emph{$\X$-seed mutation} $\mu_k$ in direction $k$ takes the labeled $\X$-seed $(\textbf{x}, B)$ in $\PP$ to the $\X$-seed $(\textbf{x}', B')$ in $\PP$ and \emph{matrix mutation} takes $B$ to $B'$. Two skew-symmetrizable integer matrices are \emph{mutation equivalent} if some sequence of matrix mutations takes one to the other.

\end{defn}
 
Note that $\mu_k(\textbf{a},\textbf{x}, B)$ is indeed another labeled $\A$-seed, as $B'$ is skew-symmetrizable and $\textbf{a}'$ again consists of algebraically independent elements generating $\mc{F}$. One can check that $\mu_k$ is an involution. 



\subsection{Seed Patterns and Exchange Graphs}

We organize all seeds obtainable from each other by a sequence of mutations in a \emph{seed pattern}. Let $\TT_n$ denote the (infinite) $n$-regular tree with edges labeled with $1, \dots, n$ so that no vertex is in two edges with  the same label.

\begin{defn} A rank $n$ \emph{$\A$-seed pattern} (respectively, \emph{$\X$-seed pattern}) $\mc{S}$ is an assignment of labeled $\A$-seeds (respectively $\X$-seeds) $\Sigma_t$ to the vertices $t$ of $\TT_n$ so that if $t$ and $t'$ are connected by an edge labeled $k$, then $\Sigma_t=\mu_k(\Sigma_{t'})$.
\end{defn}
 
Since mutation is involutive, a seed pattern $\mc{S}$ is completely determined by the choice of a single seed $\Sigma$; we write  $\mc{S}(\Sigma)$ for the seed pattern containing $\Sigma$. Note that in the language of Fomin and Zelevinsky, an $\A$-seed pattern is a ``seed pattern" and an $\X$-seed pattern is a ``$Y$-pattern".

Given an $\A$-seed pattern $\mc{S}(\textbf{a}, \textbf{x}, B)$ , one can obtain two $\X$-seed patterns. The first is $\mc{S}|_{\X}:=\mc{S}(\textbf{x}, B)$, the $\X$-seed pattern in $\PP$ obtained by simply ignoring the $\A$-clusters of every seed. The second is an $\X $-seed pattern in $\mc{F}$ whose construction is outlined in the following proposition. 

\begin{prop}[\protect{\cite[Proposition 3.9]{FZ4}}] Let $\mc{S}=\{\Sigma_t\}_{t \in \TT_n}$ be an $\A$-seed pattern in $\mc{F}$. For a seed $\Sigma_t=(\textbf{a}, \textbf{x}, B)$ of $\mc{S}$ with 
\[ \textbf{a}=(a_1, \dots, a_n), ~ \textbf{x}=(x_1, \dots, x_n), ~ \text{and } B=(b_{ij})
\]
let $\hat{\Sigma}_t=(\hat{\textbf{x}}, B)$ where $\hat{\textbf{x}}=(\hat{x}_1, \dots, \hat{x}_n)$ is the $n$-tuple of elements of $\mc{F}$ given by

\[ \hat{x}_j= x_j \prod_i a_i^{b_{ij}}.
\]

Then $\hat{\mc{S}}:=\{\hat{\Sigma}_t\}_{t \in \TT_n}$  is an $\X$-seed pattern in $\mc{F}$. In other words, if $\mu_k(\Sigma_t)=\Sigma_{t'}$ , then $\mu_k(\hat{\Sigma}_t)=\hat{\Sigma}_{t'}$. 
\end{prop}

One can think of $\hat{\mc{S}}$ as recording the ``exchange information" of $\mc{S}(\textbf{a}, \textbf{x}, B)$; indeed, the $\X$-variables of $\hat{\mc{S}}$ are rational expressions whose numerators and denominators are, up to multiplication by an element of $\PP$, the two terms on the right hand side of an exchange relation of $\mc{S}$.

Note that seed patterns may carry redundant information, in that the same seed can be assigned to multiple vertices of $\TT_n$. Further, two labeled seeds in a seed pattern may be the same up to relabeling. We remedy this by calling two labeled seeds $\Sigma=(\textbf{a}, \textbf{x}, B)$ and $\Sigma'=(\textbf{a}', \textbf{x}', B')$ equivalent (and writing $\Sigma \sim \Sigma'$) if one can obtain $\Sigma'$ by simultaneously reindexing $\textbf{a}$, $\textbf{x}$, and the rows and columns of $B$. We define an analogous equivalence relation for $\X$-seeds, also denoted $\sim$.

An \emph{$\A$-seed} in $\mc{F}$ (respectively \emph{$\X$-seed} in $\PP$) is an equivalence class of labeled $\A$-seeds in $\mc{F}$ (respectively labeled $\X$-seeds in $\PP$) with respect to $\sim$. The seed represented by the labeled seed $\Sigma$ is denoted $[\Sigma]$. We mutate a seed $[\Sigma]$ by applying a mutation $\mu_k$ to $\Sigma$ and taking its equivalence class.


\begin{defn} The \emph{exchange graph} of a seed pattern $\mc{S}$ is the ($n$-regular connected) graph whose vertices are the seeds in $\mathcal{S}$ and whose edges connect seeds related by a single mutation. Equivalently, the exchange graph is the graph one obtains by identifying the vertices $t, t'$ of $\TT_n$ such that $\Sigma_t \sim \Sigma_{t'}$.
\end{defn}

Exchange graphs were defined for $\A$-seed patterns in \cite[Definition 7.4]{FZ1}, but can equally be defined for $\X$-seed patterns as we do here. It is conjectured that the exchange graph of an $\A$-seed pattern $\mc{S}=\mc{S}(\textbf{a}, \textbf{x}, B)$ depends only on $B$ \cite[Conjecture 4.3]{FZ4}, meaning that $\mc{S}|_{\X}$ does not influence the combinatorics of $\mc{S}$. The exchange graphs of $\mc{S}|_{\X}$ and $\hat{\mc{S}}$ can be obtained by identifying some vertices of the exchange graph of $\mc{S}$, as passing to either $\X$-seed pattern preserves mutation and the equivalence of labeled seeds. It is not known in general when any pair of these exchange graphs is equal.


\section{Finite Type Seed Patterns}

We now restrict our attention to seed patterns of finite type.

\begin{defn} An $\A$-seed pattern is of \emph{finite type} if it has finitely many seeds. 
\end{defn}

Finite type seed patterns were classified completely in \cite{FZ2}; they correspond exactly to finite (reduced crystallographic) root systems, or equivalently, finite type Cartan matrices (see for example \cite[Chapter 5]{Bour}).

For a skew-symmetrizable integer matrix $B=(b_{ij})$, its \emph{Cartan counterpart} is the matrix $A(B)=(a_{ij})$ defined by $a_{ii}=2$ and $a_{ij}=-|b_{ij}|$ for $i \neq j$.

\begin{thm}[\protect{\cite[Theorems 1.5-1.7]{FZ2}}]
\begin{enumerate}[(i.)]
\item An $\A$-seed pattern is of finite type if and only if the Cartan counterpart of one of its exchange matrices is a finite type Cartan matrix. 
\item Suppose $B$, $B'$ are skew-symmetrizable integer matrices such that $A(B)$, $A(B')$ are finite type Cartan matrices. Then $A(B)$ and $A(B')$ are of the same Cartan-Killing type if and only if $B$ and $B'$ are mutation equivalent (modulo simultaneous relabeling of rows and columns.).
\end{enumerate}
\end{thm}

In light of this theorem, we refer to a finite type $\A$-seed pattern as type $A_n$, $B_n$, etc.


\begin{defn} An $\X$-seed pattern $\mc{S}$ is of \emph{type $Z_n$} if the Cartan counterpart of one of its exchange matrices is a type $Z_n$ Cartan matrix.
\end{defn}

We will call such $\X$-seed patterns \emph{Dynkin type} rather than finite type, since not all $\X$-seed patterns with finitely many seeds are of this form. For example, let $\PP$ be any tropical semifield, and $B$ the matrix

\[ \begin{bmatrix} 0 & 2\\
-2 & 0\\
\end{bmatrix}.
\] 

Then the $\X$-seed pattern $\mc{S}((1, 1), B)$ in $\PP$ has a single unlabeled seed, but the Cartan counterpart of $B$ is not finite type. In general, if $B$ is an $n \times n$ skew-symmetrizable matrix with finite mutation equivalence class, then the $\X$-seed pattern $\mc{S}((1, \dots, 1), B)$ in a tropical semifield will have finitely many seeds. Indeed, in this case, mutation of the initial seed $((1, \dots, 1), B)$ in direction $k$ results in the seed $((1, \dots, 1), \mu_k(B))$. It follows that any sequence of mutations results in a seed with every $\X$-variable equal to 1, so seeds are distinguished from each other only by their exchange matrices. As $B$ is mutation equivalent to finitely many matrices, there are only finitely many seeds in the seed pattern. 



\subsection{Triangulations for types $A$ and $D$}
The material in this section is part of a more general theory of $\A$-seed patterns from surfaces, developed in \cite{FST}.

Let $\textbf{P}_n$ denote a convex $n$-gon and $\textbf{P}^\bullet_n$ denote a convex $n$-gon with a distinguished point $p$ (a \emph{puncture}) in the interior. For $\textbf{P} \in \{\textbf{P}_n, \textbf{P}^\bullet_n\}$, the vertices and puncture of $\textbf{P}$ are called \emph{marked points}. An \emph{arc} of $\textbf{P}$ is a non-self-intersecting curve $\gamma$ in $\textbf{P}$ such that the endpoints of $\gamma$ are distinct marked points, the relative interior of $\gamma$ is disjoint from $\partial \textbf{P}\cup \{p\}$, and $\gamma$ does not cut out an unpunctured digon. An arc incident to the puncture $p$ is a \emph{radius}. Arcs are considered up to isotopy.

\begin{figure}
\centering
\includegraphics[width=0.8\textwidth]{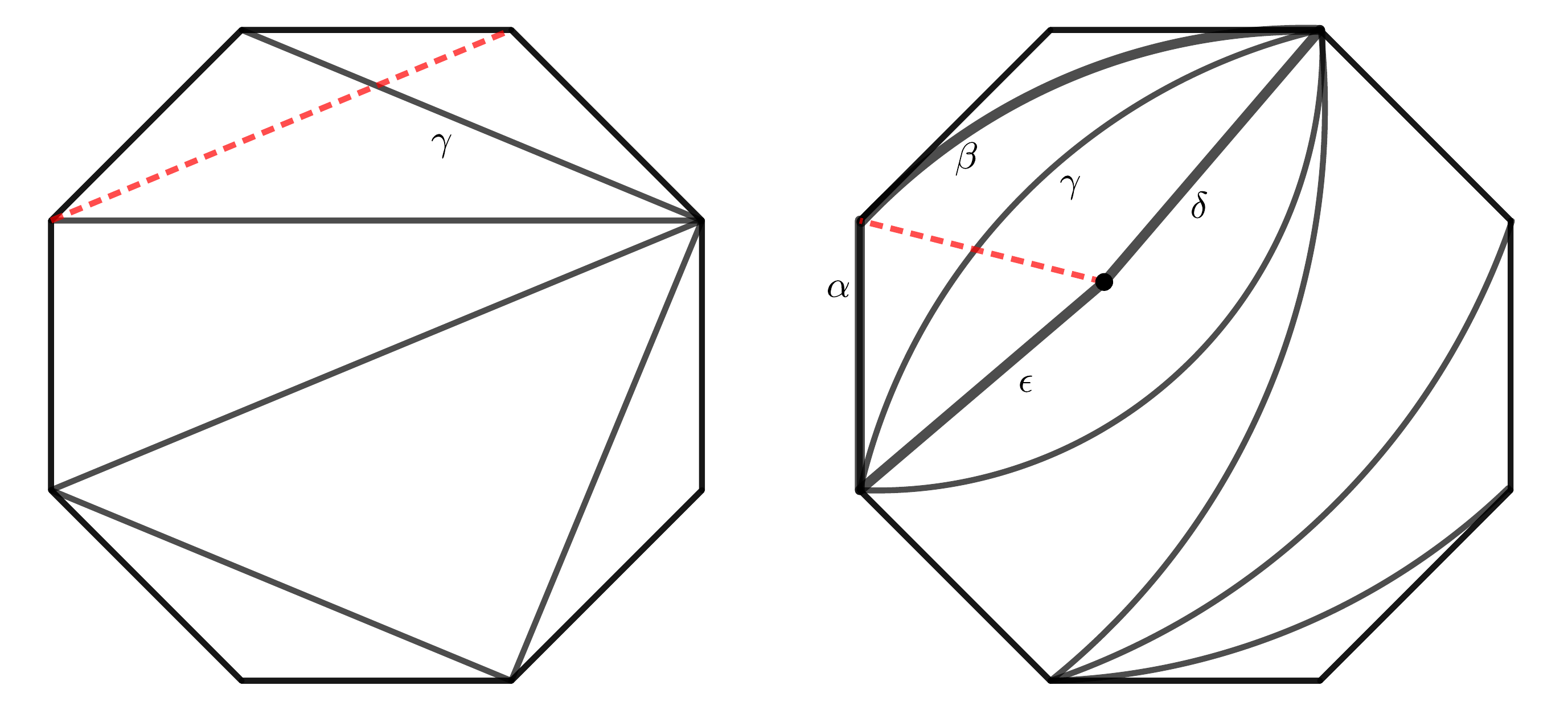}
\caption{Triangulations of $\textbf{P}_8$ and $\textbf{P}_8^\bullet$ are shown in solid lines; the dashed arc is the flip of $\gamma$. On the right, the quadrilateral $q_T(\gamma)=\{\alpha, \beta, \delta, \epsilon\}$ is shown in bold.}
\label{fig:AD}
\end{figure}

A \emph{tagged arc} of $P$ is either an ordinary arc between two vertices or a radius that is labeled either ``notched" or ``plain." Two tagged arcs $\gamma$, $\gamma'$ are \emph{compatible} if their untagged versions do not cross (or to be precise, there are two noncrossing arcs isotopic to $\gamma$ and $\gamma'$) with the following modification: if $\gamma$ is a notched radius and $\gamma'$ is plain, they are compatible if and only if their untagged versions coincide.

A \emph{tagged triangulation} $T$ is a maximal collection of pairwise compatible tagged arcs. All tagged triangulations of $\textbf{P}$ consist of the same number of arcs. 

\begin{defn}Let $T$ be a tagged triangulation, $\gamma$ an arc in $T$, and $\gamma'$ either an arc in $T$ or a boundary segment. Then $\gamma$ and $\gamma'$ are \emph{adjacent in $T$} if they are adjacent in a triangle of $T$ or if there is a third arc $\alpha$ in $T$ such that $\{\alpha, \gamma, \gamma' \}$ form a once-punctured digon with a radius (see Figure \ref{fig:qtex}). 
\end{defn}

In particular, if $T$ contains two arcs forming a once-punctured digon, the two sides of the digon are adjacent, each side is adjacent to each radius inside the digon, and the two radii inside the digon are \emph{not} adjacent.

\begin{defn} Let $T$ be a tagged triangulation, and $\gamma$ an arc in $T$. The \emph{quadrilateral} $q_{T}(\gamma)$ of an arc $\gamma$ in $T$ consists of the arcs of $T$ and boundary segments adjacent to $\gamma$ if $\gamma$ is not a radius in a once-punctured digon (see Figure \ref{fig:AD}); if $\gamma$ is a radius in a once-punctured digon, $q_{T}(\gamma)$ consists of the arcs and boundary segments adjacent to $\gamma$ together with the two radii compatible with $\gamma$ in the once-punctured digon (see Figure \ref{fig:quadex}, lower right). The arc $\gamma$ is a \emph{diagonal} of its quadrilateral.
\end{defn}

Note that if $\gamma$ is a radius in a once-punctured digon, $q_T(\gamma)$ is not part of any tagged triangulation.

The following result gives us a local move on tagged triangulations.

\begin{prop}[\protect{\cite[Theorem 7.9]{FST}}] Let $T=\{\gamma_1, \dots, \gamma_n\}$ be a tagged triangulation of $\textbf{P}$. For all $k$, there exists a unique tagged arc $\gamma_k'\neq \gamma_k$ such that $\mu_k(T):=T \setminus \{\gamma_k\} \cup \{\gamma_k'\}$ is a tagged triangulation of $\textbf{P}$.
\end{prop}

The arc $\gamma'$ in the above proposition is called the $\emph{flip}$ of $\gamma$ with respect to $T$ (or with respect to $q_T(\gamma)$, since $\gamma$ and $\gamma'$ are exactly the two diagonals of $q_T(\gamma)$). 

We define the \emph{flip graph} of $\textbf{P}$ to be the graph whose vertices are tagged triangulations of $\textbf{P}$ and whose edges connect triangulations that can be obtained from each other by flipping a single arc. The flip graph of $\textbf{P}$ is connected. 

We can encode a tagged triangulation $T=(\gamma_1, \dots, \gamma_n)$ in a skew-symmetric $n \times n$ integer matrix $B(T)$. The nonzero entries in $B(T)$ correspond to pairs of adjacent arcs; the sign of these entries records the relative orientation of the arcs.

\begin{figure}
\centering
\includegraphics[width=\textwidth]{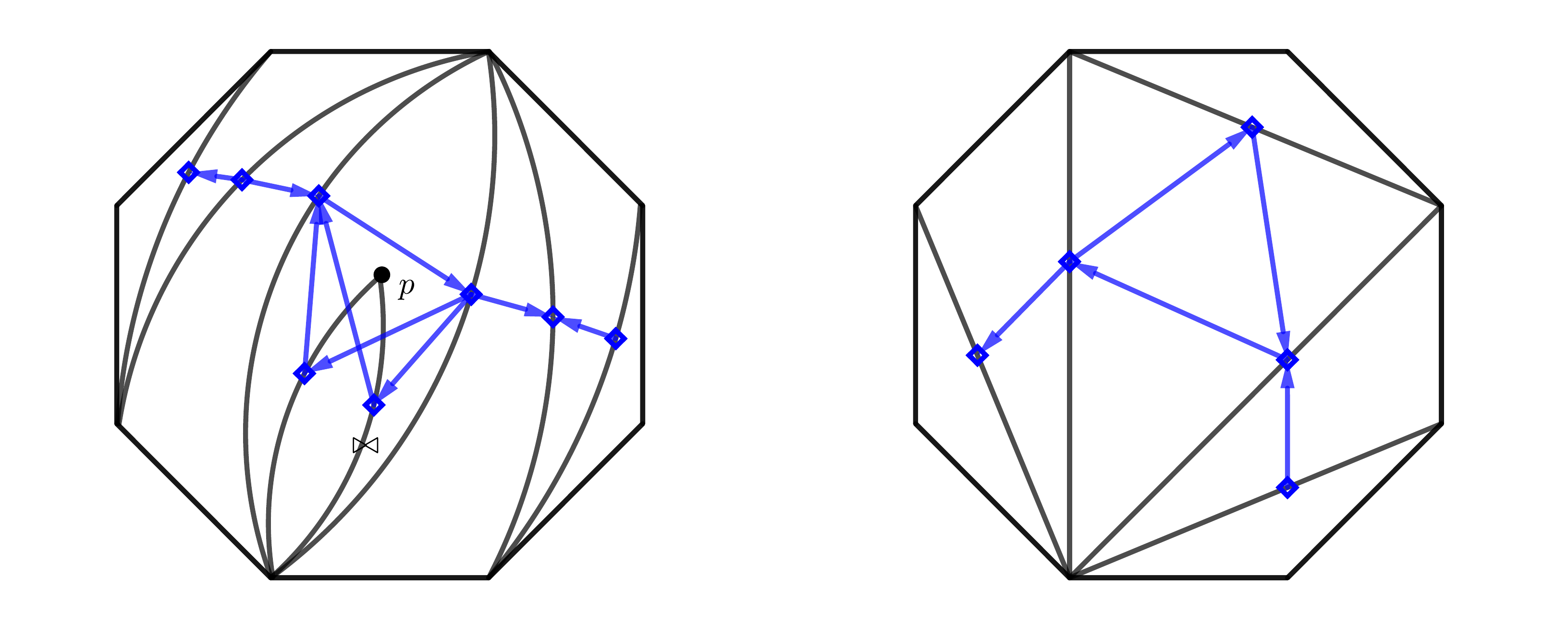}
\caption{\label{fig:qtex} Triangulations of $\textbf{P}_8$ and $\textbf{P}_8^\bullet$ and the associated quivers $Q(T)$.}
\end{figure}

To find the entries of $B(T)$, we define a quiver $Q(T)$. Place a vertex $i$ in the interior of each arc $\gamma_i$. Then put an arrow from $i$ to $j$ if $\gamma_i$ and $\gamma_j$ are two sides of a triangle in $T$ and $i$ immediately precedes $\gamma_j$ moving clockwise around this triangle, unless $\gamma_i$ and $\gamma_j$ are arcs in a triangle containing the notched and plain versions of the same radius. In this case, add arrows as shown in Figure \ref{fig:qtex}. Finally, delete a maximal collection of 2-cycles. If there is an arrow from $i$ to $j$ in $Q(T)$, we write $i \to j$.

Let $b_{ij}(T):=\#\{ \text{arrows } i \to j \text{ in } Q(T)\}- \#\{ \text{arrows } j \to i \text{ in } Q(T)\}$. We define $B(T):=(b_{ij}(T))$.

Flips of arcs are related to matrix mutation in the following way: for a tagged triangulation $T=\{\gamma_1, \dots, \gamma_n\}$ of $\textbf{P}$, $\mu_k(B(T))=B(\mu_k(T))$, or, in words, flipping $\gamma_k$ changes $B(T)$ by mutation in direction $k$.

As the following theorem shows, these triangulations entirely encode the combinatorics of type $A$ and $D$ $\A$-seed patterns. 

\begin{thm}[\cite{FST}] \label{ADtriangulation} Let $\textbf{P}=\textbf{P}_{n+3}$ (resp. $\textbf{P}=\textbf{P}_n^\bullet$). Consider an $\A$-seed pattern $\mc{S}$ such that some exchange matrix is $B(T_0)$ for some triangulation $T_0$ of $\textbf{P}$. Then $\mc{S}$ is type $A_n$ (resp. $D_n$) and there is a bijection $\gamma \mapsto a_\gamma$ between arcs of $P$ and $\A$-variables of $\mc{S}$. Further, if $\Sigma=(\textbf{a}, \textbf{x}, B)$ is a seed of $\mc{S}$, there is a unique triangulation $T$ such that $\textbf{a}=\{a_{\gamma}\}_{\gamma \in T}$ and $B=B(T)$. Finally, mutation in direction $k$ takes the seed corresponding to $T$ to the seed corresponding to $\mu_k(T)$, implying that the exchange graph of $\mc{S}$ is isomorphic to the flip graph of $P$.
\end{thm}

\subsection{Triangulations for types $B$ and $C$}

To obtain triangulations whose adjacency matrices are exchange matrices of type $B_n$ and $C_n$ $\A$-seed patterns, we ``fold" triangulations of $\textbf{P}_{2n+2}$ and $\textbf{P}_{n+1}^\bullet$. This is part of a larger theory of folded cluster algebras (see \cite[Chapter 4]{book}).

 Let $G=\ZZ/2\ZZ$. We write $\textbf{P}^{~G}_{2n}$ for $\textbf{P}_{2n}$ equipped with the $G$-action taking vertex $i$ to vertex $i':=i+n$ (with labels considered modulo $2n$). This induces an action of $G$ on the arcs of $\textbf{P}_{2n}$. The triangulations of $\textbf{P}^{~G}_{2n}$ are defined to be the triangulations of $\textbf{P}_{2n}$ fixed by the $G$-action, commonly called \emph{centrally symmetric} triangulations. 

\begin{figure}
\centering
\includegraphics[width=0.8\textwidth]{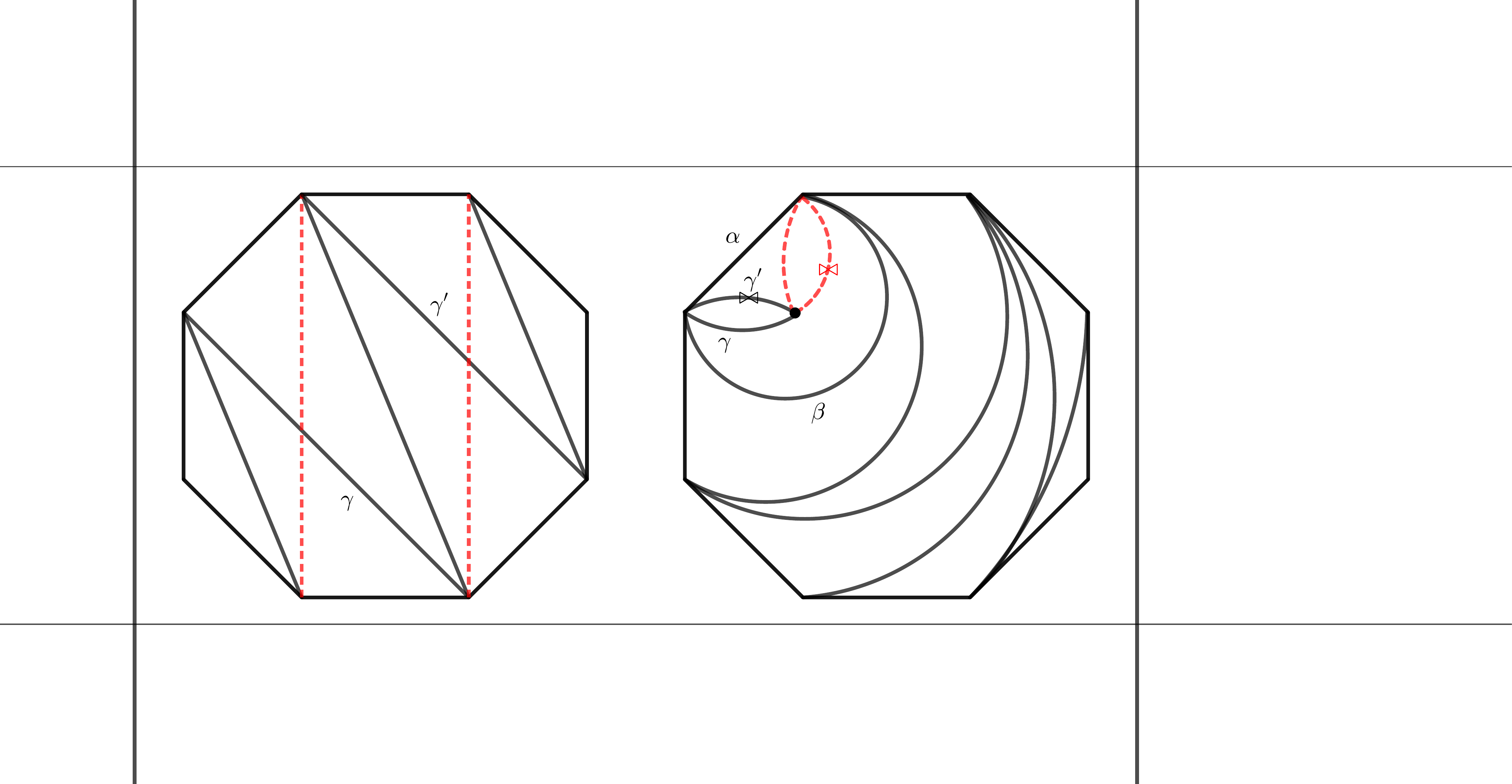}
\caption{Triangulations of $\textbf{P}_8^G$ and $\textbf{P}_8^{\bullet ~G}$ are shown in solid lines. The orbit $[\gamma]$ of $\gamma$ is $\{\gamma, \gamma'\}$ and the flip of $[\gamma]$ is dashed. On the right, $q_T([\gamma])=\{\alpha, \beta, \gamma, \gamma'\}$. }
\label{fig:BC}
\end{figure}

We write $\textbf{P}_{n}^{\bullet \ G}$ for $\textbf{P}_{n}^\bullet$ equipped with the $G$-action switching the notched and plain version of a radius. Again, the triangulations of $\textbf{P}_{n}^{\bullet \ G}$ are defined to be the triangulations of $\textbf{P}_{n}^{\bullet}$ fixed under the $G$-action, which are exactly those triangulations containing both the notched and plain versions of the same radius. 

Let $\textbf{P} \in \{\textbf{P}_{2n}, \textbf{P}_{n}^\bullet\}$, and let $T$ be a triangulation of $\textbf{P}^G$. For $\gamma \in T$, let $[\gamma]$ denote the $G$-orbit of $\gamma$. The \emph{quadrilateral} of $[\gamma]$, denoted $q_T([\gamma])$, is all of the arcs and boundary segments adjacent to arcs in $[\gamma]$ (which is exactly the $G$-orbit of $q_T(\gamma)$ as long as $\gamma$ is not a radius in a once-punctured digon). Flipping the arcs in $[\gamma]$ results in another triangulation of $\textbf{P}^G$ (which does not depend on the order of arc flips, since arcs in the same orbit are pairwise not adjacent). We define the flip graph of $\textbf{P}^{G}$ in direct analogy to that of $P$; again, it is connected.

We associate to each triangulation $T$ of $\textbf{P}^G$ a skew-symmetrizable integer matrix $B^G(T)$, whose rows and columns are labeled by $G$-orbits of arcs of $T$. Let $T=\{\gamma_1, \dots, \gamma_m\}$, and let $I, J$ be the indices of arcs in the $G$-orbit of $\gamma_i$ and $\gamma_j$. Then

\begin{equation} \label{eq:foldedexch}b^G_{IJ}(T)= \sum_{i \in I} b_{ij}(T)
\end{equation}

where $B(T)=(b_{ij}(T))$ is the usual signed adjacency matrix of $T$. In terms of $Q(T)$, $b^G_{IJ}(T)$ is the total number of arrows from all representatives of $I$ to a fixed representative of $J$. The entries of $B^G(T)$ are well-defined: for $g \in G$, if $g(\gamma_i)=\gamma_i'$ and $g(\gamma_j)=\gamma_j'$, then $b_{ij}=b_{i'j'}$, so the value of $b_{IJ}^G$ does not change if a different element of $J$ is used to compute \eqref{eq:foldedexch}.

Just as with usual triangulations, (orbits of) arc flips and matrix mutation interact nicely: if $T$ is a triangulation containing arc $\gamma$, flipping the arcs in $[\gamma]$ corresponds to mutating $B^G(T)$ in the direction labeled by $[\gamma]$. Further, we have the following theorem.

\begin{thm}[\protect{\cite[Section 5.5]{book}}] \label{BCtriangulation} Let $\textbf{P}=\textbf{P}_{n+1}^\bullet$ (resp. $\textbf{P} =\textbf{P}_{2n+2}$). Consider an $\A$-seed pattern $\mc{S}$ such that some exchange matrix is $B^G(T_0)$ for some triangulation $T_0$ of $\textbf{P}^G$. Then $\mc{S}$ is type $B_n$ (resp. $C_n$) and there is a bijection $[\gamma] \mapsto a_{[\gamma]}$ between arcs of $\textbf{P}$ and $\A$-variables of $\mc{S}$. Further, if $\Sigma=(\textbf{a}, \textbf{x}, B)$ is a seed of $\mc{S}$, there is a unique triangulation $T$ such that $\textbf{a}=\{a_{[\gamma]}\}_{\gamma \in T}$ and $B=B^G(T)$. Finally, mutation in direction $K$ takes the seed corresponding to $T$ to the seed corresponding to $\mu_K(T)$, implying that the exchange graph of $\mc{S}$ is isomorphic to the flip graph of $\textbf{P}^G$.
\end{thm}


\section{Dynkin type $\X$-seed patterns}

Let $\mc{S}$ be an $\X$-seed pattern of type $Z_n$ ($Z\in \{A, B, C, D\}$) over an arbitrary semifield $\PP$, and let $\X(\mc{S})$ denote the set of $\X$-variables of $\mc{S}$. Let $\textbf{P}$ be the surface whose triangulations encode the combinatorics of type $Z_n$ $\A$-seed patterns ($\textbf{P}=\textbf{P}_{n+3}$ for $Z_n=A_n$, $\textbf{P}=\textbf{P}_n^\bullet$ for $Z_n=D_n$, $\textbf{P}=\textbf{P}^{~G}_{2n+2}$ for $Z_n=C_n$, $\textbf{P}=\textbf{P}_{n+1}^{\bullet ~ G}$ for $Z_n=B_n$). In this section, we relate the $\X$-variables of $\mc{S}$ to the triangulations of $\textbf{P}$, and show a bijection between $\X(\mc{S})$ and quadrilaterals (with a choice of diagonal) of $\textbf{P}$ in the case when $\PP$ is the universal semifield. Note that in what follows, ``arc" should usually be understood to mean ``orbit of arc" if $\mc{S}$ is type $B$ or $C$. 

 First, notice that Theorems \ref{ADtriangulation} and \ref{BCtriangulation} imply that one can associate to each triangulation of $\textbf{P}$ a seed of $\mc{S}$ such that mutation of seeds corresponds to flips of arcs. Indeed, consider any $\A$-seed pattern $\mc{R}$ with $\mc{R}|_{\X}=\mc{S}$; if the triangulation $T$ corresponds to the $\A$-seed $(\textbf{a}, \textbf{x}, B)$ with arc $\gamma_k$ corresponding to $\A$-variable $a_k$, then we associate to $T$ the $\X$-seed $(\textbf{x}, B)$ and to the arc $\gamma_k$ the $\X$-variable $x_k$. We write $\Sigma_T$ to indicate this association, and write $x_{T, \gamma}$ for the $\X$-variable associated to arc $\gamma$ in $\Sigma_T$.
 
Note that \emph{a priori} two distinct triangulations may be associated to the same $\X$-seed, and an $\X$-variable may be associated to a number of different arcs (though we will see that this is not the case). Further, an arc may be associated to different $\X$-variables in different triangulations. 
 
 The next observation follows immediately from the definition of seed mutation.
 
 \begin{rmk} \label{rmk:local} Consider an $\X$-seed $(\textbf{x}, B)$, and let $B'=(b'_{ij})$ be the mutation of $B$ in direction $k$. For $j \neq k$, if $b_{jk}=0$, then mutating at $k$ will not change $x_j$. Further, $b'_{ji}=b_{ji}$ and $b'_{ij}=b_{ij}$ for all $i$, since the skew-symmetrizability of $B$ implies $b_{kj}=0$ as well. Thus, if $k_1, \dots, k_t$ are indices such that $b_{k_{s}j}=0$, then the mutation sequence $\mu_{k_t}\circ\cdots \circ \mu_{k_2} \circ \mu_{k_1}$ leaves $x_j$ unchanged.
\end{rmk}

In other words, consider $x_{T, \gamma}$, an $\X$-variable in $\Sigma_T$. Any sequence of flips of arcs \emph{not} in $q_T(\gamma) \cup \{\gamma\}$ will result in a triangulation $S$ such that $\gamma \in S$ and $x_{S, \gamma}=x_{T, \gamma}$.

\begin{prop} \label{surjection} Let $Q'=\{q_T(\gamma) \cup \{\gamma\} | ~T \text{ a triangulation of }\textbf{P}, \gamma \in T\}$ be the set of quadrilaterals (with a choice of diagonal) of $\textbf{P}$. Then the following map is a surjection: 
\begin{align*}
f:Q' &\to \X(\mc{S})\\
q_T(\gamma) \cup \{\gamma\} &\mapsto x_{T, \gamma}
\end{align*}.
\end{prop}

We remark that this proposition in fact holds in the generality of $\X$-seed patterns from marked surfaces, by the discussion following Proposition 9.2 in \cite{FST}. If one considers a collection $A$ of compatible tagged arcs of a marked surface $(S, M)$, the simplicial complex of tagged arcs compatible with $A$ is the tagged arc complex of another (possibly disconnected) surface $(S', M')$, so its dual graph (that is, the flip graph of $(S', M')$) is connected by \cite[Proposition 7.10]{FST}. This implies that any two triangulations containing $A$ are connected by a series of flips of arcs not in $A$. However, Proposition \ref{surjection} is not hard to see directly in this case, so we provide a proof here.

\begin{proof} If $f$ is well-defined, it is clearly surjective. By Remark \ref{rmk:local}, to show $f$ is well-defined, it suffices to show that all triangulations $T$ with $\gamma \in T$ and $q_T(\gamma)=q$ can be obtained from one another by flipping arcs not in $q \cup \{\gamma\}$.

Observe that removing the interior of $q$ from $\textbf{P}$ gives rise to a new surface $\textbf{P}'$, whose connected components are polygons, once-punctured polygons, or unions of several of these that intersect only at vertices. The triangulations of $\textbf{P}'$ do not include the arcs in $q$, and the flip graph of $\textbf{P}'$ is connected. There is a bijection from triangulations $T'$ of $\textbf{P}'$ to triangulations of $\textbf{P}$ containing $q \cup \{\gamma\}$: $T' \mapsto T' \cup q \cup \{\gamma\}$ (up to changing the tagging of radii in $T$). This bijection respects flips, so this shows the desired result.
\end{proof}

We have the immediate corollary:

\begin{cor} Let $\mf{q}(\textbf{P})$ denote the number of quadrilaterals of $\textbf{P}$. Then $|\X(\mc{S})|\leq 2\mf{q}(P)$.
\end{cor}

The above statements hold regardless of the choice of $\PP$ and initial $\X$-cluster. However, $|\X(\mc{S})|$ can vary for different choices of $\PP$ and the $\X$-cluster of a fixed seed, as the following example shows.

\begin{ex}\label{ex:degenerateseed} Let $B$ be the following matrix with Cartan counterpart of type $A_3$:

\[ \begin{bmatrix} 0 & 1 & 0\\
-1 & 0& 1\\
0 & -1 & 0\\
\end{bmatrix}.
\] 
 
If $\PP$ is any tropical semifield, then the $\X$-seed pattern $\mc{S}((1, 1, 1), B)$ in $\PP$ has a single $\X$-variable, which is equal to 1. 

If $\PP=\QQ_{sf}(x_1, x_2, x_3)$, the $\X$-seed pattern $\mc{S}_1:=\mc{S}((1/x_2, x_1/x_3, x_2), B)$ in $\PP$ has fewer $\X$-variables than $\mc{S}_2:=\mc{S}((x_1, x_2, x_3), B)$. For example, in $\mc{S}_1$, we have

\[\left( \left (\frac{1}{x_2},\frac{x_1}{x_3}, x_2 \right ), B\right) \xrightarrow{\mu_1} \left (\left(x_2, \frac{x_1}{x_3(1+x_2)}, x_2\right), \mu_1(B)\right).
\]

Note that $x_2$ appears twice in $\mu_1((1/x_2, x_1/x_3, x_2), B)$. 

In $\mc{S}_2$, we have

\[\left( \left (x_1,x_2, x_3 \right ), B\right) \xrightarrow{\mu_1} \left (\left(\frac{1}{x_1}, \frac{x_1x_2}{(1+x_1)}, x_3\right), \mu_1(B)\right).
\]

So $|\X(\mc{S}_1)|\leq |\X(\mc{S}_2)|-1$.
 Using a computer algebra system, one can check that $|\X(\mc{S}_1)|= 18$, while $|\X(\mc{S}_2)|=30$.
\end{ex}

If we fix an exchange matrix $B$ and allow $\PP$ and the $\X$-cluster of the seed $(\textbf{x}, B)$ to vary, $\mc{S}(\textbf{x}, B)$ will have the largest number of $\X$-variables when $\PP=\QQ_{sf}(t_1, \dots, t_n)$ and $\textbf{x}$ consists of elements that are algebraically independent over $\QQ(t_1, \dots, t_n)$. Indeed, let $\mc{S}_{sf}$ be such a seed pattern, and $\mc{S}$ an arbitrary seed pattern in a semifield $\PP$ containing the exchange matrix $B$. The $\X$-variables of $\mc{S}$ can be obtained from the $\X$-variables of $\mc{S}_{sf}$ by replacing ``$+$" with ``$\oplus$" and evaluating at the appropriate elements of $\PP$, so we have $|\X(\mc{S})| \leq |\X(\mc{S}_{sf})|$. As can be seen in Example \ref{ex:degenerateseed}, the assumption that the elements of $\textbf{x}$ are algebraically independent is necessary. Note also that if the elements of $\textbf{x}$ are algebraically independent, so are the elements of an arbitrary $\X$-seed in $\mc{S}_{sf}$, since $\X$-seed mutation in this case simply has the effect of multiplying the $\X$-variables by rational functions.

In light of this observation, we now focus on $\X(\mc{S}_{sf})$. To show that the surjection $f$ of Proposition \ref{surjection} is a bijection for $\mc{S}_{sf}$, it suffices to show that $f$ is injective for some $\X$-seed pattern of type $Z_n$. To do this, we use specific examples of $\A$-seed patterns of type $Z_n$ given in \cite[Chapter 5]{book} 
and \cite[Section 12.3]{FZ2},
 which we denote by $\mc{R}(Z_n)$. These seed patterns are over a tropical semifield $\PP$; the $\A$-variables and the generators of $\PP$ are SL$_2$- or SO$_2$-invariant polynomials in the entries of a $2 \times m$ matrix.

\begin{prop} \label{prop:injection} Consider the $\A$-seed pattern $\mc{R}=\mc{R}(Z_n)$. Then the surjection $f:Q' \to \X(\hat{\mc{R}})$ of Proposition \ref{surjection} is injective.
\end{prop}

We delay the description of $\mc{R}(Z_n)$ and the proof of this proposition to Section \ref{sec:geocontruct}.

Theorem \ref{quadbijection} is an immediate corollary of Proposition \ref{prop:injection}, as is the number of $\X$-variables in $\mc{S}_{sf}$.

\begin{cor} $|\X(\mc{S}_{sf})|= 2\mf{q}(\textbf{P})$.
\end{cor}

\subsection{Quadrilateral counts}

We label the vertices of $\textbf{P} \in \{\textbf{P}_n, \textbf{P}_n^\bullet, \textbf{P}_n^{~ G}, \textbf{P}_n^{\bullet G}\}$ clockwise with $1, \dots, n$.

\begin{prop}\label{quadnumber} The number of quadrilaterals of each surface $\textbf{P}$ are listed in the following table.

\begin{center}
\begin{tabular}{| c | c | c| c | }
 \hline
   $\textbf{P}$ & $\textbf{P}_{n+3}$ & $\textbf{P}_{2n+2}^{~G}, \textbf{P}_{n+1}^{\bullet G}$ & $\textbf{P}_{n}^\bullet$  \\ \hline
  $\mf{q}(\textbf{P})$ & $\binom{n+3}{4}$ & $\frac{1}{6}n(n+1)(n^2+2)$ & $\frac{1}{6}n(n-1)(n^2+4n-6)$ \\ \hline
\end{tabular}
\end{center}
\end{prop}

For $\textbf{P} \neq \textbf{P}_{2n+2}^{~G}$, $\mf{q}(\textbf{P})$ follow from a fairly straightforward inspection of the triangulations of the appropriate surfaces. It is clear that $\mf{q}(\textbf{P}_{n+3})=\binom{n+3}{4}$, as quadrilaterals in a polygon are uniquely determined by their vertices.

\begin{figure} 
\centering
\includegraphics[height=.4\textheight]{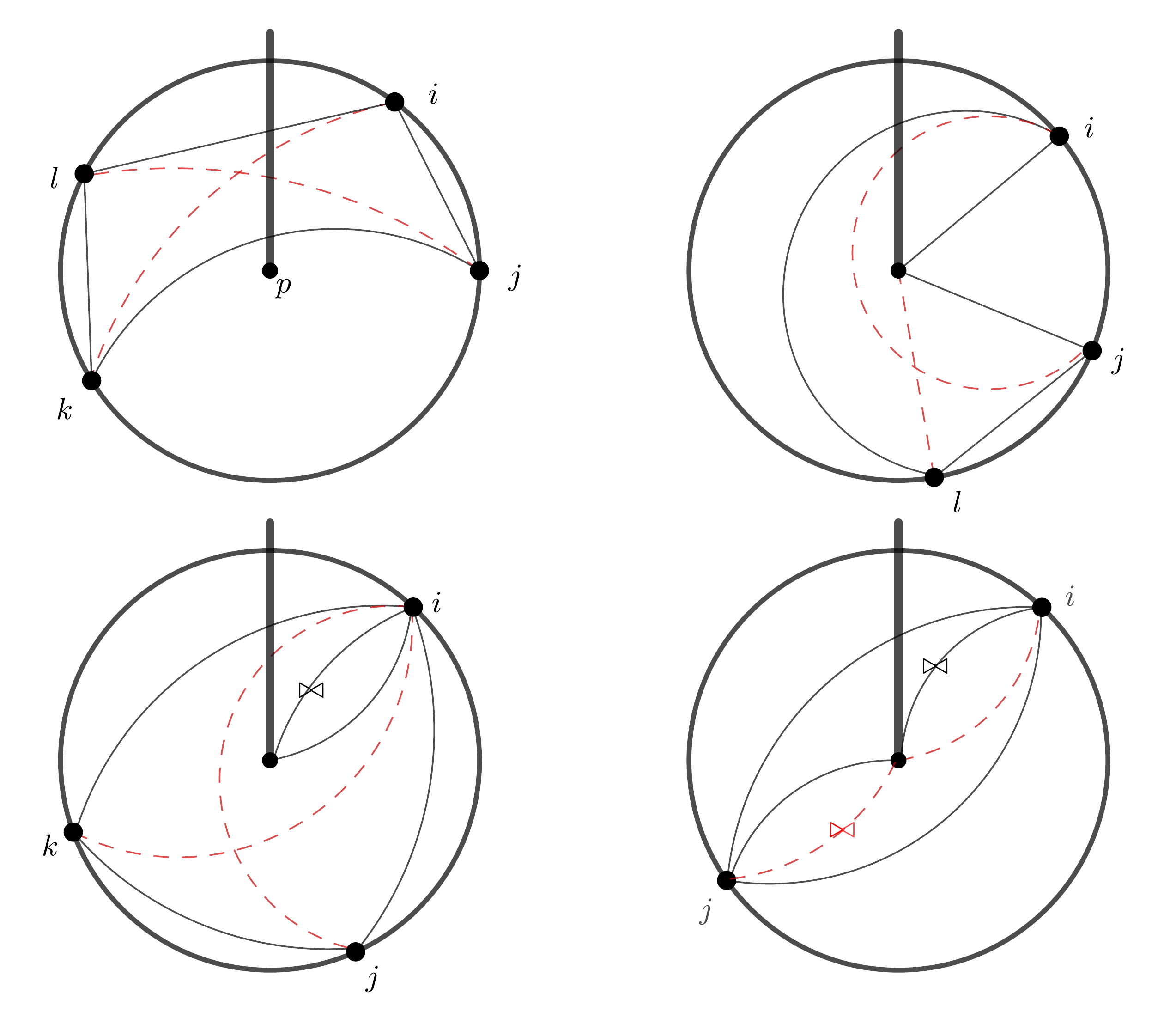}
\caption{\label{fig:quadex} The 4 types of quadrilaterals in a triangulation of $\textbf{P}_n^\bullet$, listed in clockwise order from the upper left. The arcs of the quadrilateral are solid; diagonals are dashed. The cut, the segment from $p$ to the boundary arc between vertices $1$ and $n$, is shown in bold.}
\end{figure}

\begin{prop}$\mf{q}(\textbf{P}_n^\bullet)=4 \binom{n}{4}+9\binom{n}{3}+2\binom{n}{2}=\frac{1}{6}n(n-1)(n^2+4n-6)$
\end{prop}

\begin{proof} We first introduce a ``cut" from $p$ to the boundary arc between vertices 1 and $n$ (see Figure \ref{fig:quadex}). For any two vertices $i$ and $j$, there are two distinct arcs between $i$ and $j$, one crossing the cut and the other not. For convenience, we here refer to boundary segments between adjacent vertices as ``arcs,'' though they are not arcs of any triangulation.

Let $q$ be a quadrilateral in some triangulation of $\textbf{P}_n^\bullet$. There are 4 possibilities (see Figure \ref{fig:quadex}):
\begin{enumerate}[(i)]
\item The quadrilateral $q$ consists of arcs between four vertices of $S$; the diagonals of $q$ are arcs between vertices. Then $q$ cannot enclose the puncture, since its diagonals are arcs between vertices. This places restrictions on the possible configurations of arcs. If $1 \leq i<j<k<l\leq n$ are the vertices of $q$, then either no arcs cross the cut or the arc from $i$ to $l$ and one other arc must cross the cut. This gives four quadrilaterals for each choice of four vertices.
\item The quadrilateral $q$ consists of arcs between three vertices and $p$; one diagonal is a radius and the other is an arc between vertices. Then $q$ must consist of two radii and two arcs between vertices. For any choice of three vertices $1\leq i<j<k\leq n$, there is a unique triangle on these vertices enclosing the puncture: the arc from $i$ to $k$ must cross the cut and the other two arcs must not. The quadrilateral $q$ can include any two of the arcs of this triangle. Once the two arcs are chosen, the radii in $q$ are determined. The radii must either be both notched or both unnotched, six choices in all. So there are $6\binom{n}{3}$ of these quadrilaterals.
\item The quadrilateral $q$ consists of arcs between three vertices and $p$; the diagonals are arcs between vertices. Then $q$ must enclose $p$, and the arcs involving $p$ must be the notched and unnotched version of the same radius (otherwise the quadrilateral is of type (ii)). Once the three vertices are chosen, the arcs between vertices are determined. There are three choices for the radii, so $3\binom{n}{3}$ quadrilaterals total.
\item The quadrilateral $q$ consists of arcs between two vertices and $p$; the diagonals are radii. Once the two vertices $i, j$ have been selected, the arcs between vertices are determined; they must be the two arcs between $i$ and $j$. This quadrilateral also contains 2 radii which are not compatible, for which there are 2 choices.
\end{enumerate}

\end{proof}

 \begin{prop}$ \mf{q}(\textbf{P}_{n+1}^{\bullet G})=4 \binom{n+1}{4}+3 \binom{n+1}{3}+ \binom{n+1}{2}=\frac{1}{6}n(n+1)(n^2+2)$.
 \end{prop}
 
 \begin{proof}
Recall that the triangulations of $\textbf{P}_{n+1}^{\bullet G}$ are triangulations of $\textbf{P}_{n+1}^\bullet$ using the notched and unnotched versions of the same radius; these two arcs are identified. To count quadrilaterals, we modify our count of quadrilaterals of $\textbf{P}_{n+1}^\bullet$. In triangulations of $\textbf{P}_{n+1}^{\bullet G}$, there are no quadrilaterals of type (ii), and there are fewer quadrilaterals of type (iv). In particular, these quadrilaterals are completely determined by a choice of two vertices. There are the same number of quadrilaterals of types (i) and (iii).
 
 \end{proof}

We now consider $\textbf{P}=\textbf{P}_{2n+2}^{~G}$. For $i \in \{1, \dots, 2n+2\}$, let $i':=i+n+1$ (considered mod $2n+2$).

\begin{defn}\label{defn:halfdisk} A \emph{diameter} of $\textbf{P}_{2n+2}^{~G}$ is an arc from vertex $i$ to vertex $i'$, denoted by $\delta_i$. The (closures of) half-disks resulting from removing $\delta_i$ from $\textbf{P}_{2n+2}^{~G}$ are \emph{(closed) $\delta_i$-half-disks}. Two vertices $j, k$ are \emph{separated} by a diameter $\delta_i$ if they do not lie in the same closed $\delta_i$-half-disk.
\end{defn}

 \begin{prop} Let $Q$ be the set of quadrilaterals of $\textbf{P}_{2n+2}$, $Q_1$ be the set of quadrilaterals contained in a closed $\delta_i$-half-disk some $i$, and $Q_2$ the subset of $Q\setminus Q_1$ of quadrilaterals whose vertices are not $\{i, j, i', j'\}$. There is a bijection $\alpha:Q_1 \to Q_2$.
 \end{prop}
  
\begin{proof}

We consider $\textbf{P}_{2n+2}$ as a disk with $2n+2$ vertices on the boundary. Let $i, j, k$ be vertices. We denote by $i|j$ a segment of the boundary with endpoints $i$ and $j$. To specify a particular segment, we give an orientation (e.g. ``clockwise segment $i|j$") or an interior point $k$ (e.g. $i|k|j$). If a number of interior points are given, they are listed respecting the orientation of the segment. If a quadrilateral is in $Q_1$, we list its vertices in clockwise order, starting from any point in a half-disk disjoint from the quadrilateral. Note that in this order, a quadrilateral $(a, b, c, d)\in Q_1$ is contained in a closed $\delta_a$-half-disk and a closed $\delta_d$-half-disk. If the quadrilateral is not in $Q_1$, we list its vertices in clockwise order from an arbitrary starting point.
 
 Let $(a, b, c, d)\in Q_1$. The map $\alpha$ ``flips" $b$ to $b'$, sending $(a, b, c, d)$ to $(a, c, d, b')$. 
 %
The map $\alpha$ is well-defined, as $\{a, c, d, b'\} \neq \{i, i', j, j'\}$ for any $i, j$ and for all $i\in \{a, b, c, d\}$ no closed $\delta_i$-half-disk contains all 4 vertices of $\alpha(a, b, c, d)$. 
 
We now show that $\alpha$ has an inverse, similarly given by flipping a vertex. Let $q=(a, b, c, d) \in Q_2$. Note that flipping a vertex of $q$ results in a quadrilateral in $Q_1$ only if it is the only vertex in some half-disk; one may take this half-disk to be a $\delta_i$-half-disk for some $i \in \{a, b, c, d\}$.

Suppose for some $i \in \{a, b, c, d\}$, $i'$ is also in $\{a, b, c, d\}$. Wlog, $i=a$ and thus $i'=c$ (otherwise $q$ would be contained in a half-disk). Flipping $a$ or $c$ would result in 3 vertices, so they cannot be flipped to produce an element of $Q_1$. Since $Q_2$ does not contain quadrilaterals of the form $(i, j, i', j')$, $b' \neq d$. The clockwise segment $a|c$ necessarily includes $d'$, since for all $i, j$, the segment $i|i'$ includes exactly one of $j$ and $j'$. So this segment is either $a|d'|b|c$ or $a|b|d'|c$; the clockwise segment $c|a$ is $c|d|b'|a$ or $c|b'|d|a$, respectively. In the first case, flipping $d$ gives a preimage of $q$, and flipping $b$ does not; vice-versa in the second case. 

Now, suppose no vertices in $q$ are flips of each other. There are two vertices of $q$ in one $\delta_a$-half-disk and one, say $i$, in the other. Similarly, there are two vertices of $q$ in one $\delta_i$-half disk and one, say $j$, in the other.

We claim that $i$ and $j$ are the only vertices alone in $\delta_k$-half-disks for $k\in \{a, b, c, d\}$. 

If $j=a$, then the segments $a|i|a'$ and $i|a|i'$ contain no other vertices of $q$, implying that $i'|a|i|a'$ is a segment containing no other vertices of $q$. Let $k$ be one of the other vertices of $q$. Either $k$ or $k'$ must lie on $i'|a|i|a'$ between $a$ and $a'$, and between $i$ and $i'$. Thus $k'$ in fact lies between $a$ and $i$ on this segment, so $\delta_k$ separates $a$ and $i$. Since one $\delta_k$-half-disk contains a single vertex, the statement follows.

 If $j \neq a$, then the segments $a|i|a'$ and $i|j|i'$ contain no other vertices of $q$, so $a|i|a'|j|i'$ is a segment containing no other vertices of $q$. Further, $j'$ must lie somewhere on $a|i|a'$ and cannot lie on $i|j|i'$, implying that $a|j'|i|a'|j|i'$ is a segment. Since no other vertices of $q$ appear in this segment, $i$ is the single vertex in one $\delta_j$-half-disk. Let $k$ be the other vertices of $q$. As in the above case, one of $k$ and $k'$ must lie between $a$ and $a'$, and $i$ and $i'$. Thus, $k'$ lies between $i$ and $a'$ in the arc $a|i|a'|j|i'$, implying $\delta_k$ separates $i$ and $j$. The statement follows.
 
 As noted above, flipping $i$ will result in a quadrilateral $q'$ in $Q_1$. The vertex $i'$ will not be the first or last vertex in $q'$, as by assumption the vertices of $q$ are not contained in a $\delta_i$-half-disk. However, $j$ will be, since all vertices of $q'$ are contained in one $\delta_j$-half-disk. The vertex of $q'$ closest to $j$ is $i'$, as if there is a vertex $k$ between $i'$ and $j$, then $j$ is not alone in a $\delta_i$-half-disk. So flipping $i$ gives a preimage of $q$ only if $i|j|i'$ is a clockwise arc. Reversing the roles of $i$ and $j$ in the above argument yields that flipping $j$ gives a preimage of $q$ only if $j|i|j'$ is clockwise (i.e. $i|j|i'$ is counterclockwise).
 \end{proof}
 
 Since $|Q|=|Q_1|+|Q_2|+ |\{q \in Q: q=(i, j, i', j') \text{ for } 1 \leq i<j\leq n+1\}|$, we have the following corollary.
 
 \begin{cor} $|Q_1|=\frac{1}{2}(\binom{2n+2}{4}-\binom{n+1}{2})$
 \end{cor}

 Note that the quadrilaterals of centrally symmetric triangulations of $\textbf{P}_{2n+2}$ are either in $Q_1$ or have vertex set $\{i, j, i', j'\}$ for $1\leq i< j \leq n+1$. The latter quadrilaterals are fixed by the action of $G$; the quadrilaterals in $Q_1$ have $G$-orbits of size two. So we have that
 
   \begin{equation} \label{eq:typeCquad}\mf{q}(\textbf{P}_{2n+2}^{~G})=\frac{1}{2}|Q_1|+ \binom{n+1}{2}=\frac{1}{6}n(n+1)(n^2+2).\end{equation}

\subsection{Exceptional types}

Let $Z \in \{E_6, E_7, E_8, F_4, G_2\}$ and let $\mc{S}_{sf}$ be a type $Z$ $\X$-seed pattern over $\QQ_{sf}$ with one (equivalently every) $\X$-cluster consisting of algebraically independent elements. The value of $|\X(\mc{S}_{sf})|$ was computed using a computer algebra system (Mathematica), by generating all possible $\X$-seeds via mutation.\footnote{The code is available at \url{https://math.berkeley.edu/~msb}.}


\section{Proof of Proposition \ref{prop:injection}} \label{sec:geocontruct}

We proceed type by type. The general recipe is as follows. As usual, for non-simply laced types, ``arc" should be read as ``orbit of arc." 

Let $V$ be a vector space and $\textbf{k}$ the field of rational functions on $V$. For $\textbf{P} \in \{\textbf{P}_{n+3},\textbf{P}_{n}^\bullet, \textbf{P}_{n+1}^{\bullet G}, \textbf{P}_{2n+2}^{~ G} \}$, we assign a function $P_\gamma \in \textbf{k}$ to each arc and boundary segment $\gamma$ of $\textbf{P}$. 

Let $T$ be a triangulation of $\textbf{P}$. We define seeds $\Sigma_T=(\textbf{a}, \textbf{x}, B(T))$, where the clusters and exchange matrix are indexed by arcs in $T$. The matrix $B(T)$ is the exchange matrix as defined in Sections 3.1 and 3.2, whose entries are based on $Q(T)$. The $\A$-variable associated to $\gamma \in T$ is just $P_\gamma$. 

To describe the $\X$-variables, we first extend our construction of $Q(T)$ to the boundary segments of $\textbf{P}$ to create the quiver $\overline{Q}(T)$. That is, we place a vertex at the midpoint of every arc and boundary segment, put an arrow from $\gamma$ to $\gamma'$ if $\gamma$ precedes $\gamma'$ moving clockwise around some triangle of $T$, and delete a maximal collection of 2-cycles (making slight modifications for punctured polygons). If there is an arrow from $\gamma$ to $\gamma'$ in $\overline{Q}(T)$, we write $\gamma \to \gamma'$. For simply laced types, the $\X$-variable associated to $\gamma$ is 

\begin{equation} \label{eqn:xvar}
\frac{\prod_{\gamma' \in \partial \textbf{P}} P_{\gamma'}^{\# \{\text{arrows } \gamma' \to \gamma\}}}{\prod_{\gamma' \in \partial \textbf{P}} P_{\gamma'}^{\# \{\text{arrows } \gamma \to \gamma'\}}}
.
\end{equation}

For non-simply laced types, the $\X$-variable associated to $[\gamma]$ is

\[\frac{\prod_{\gamma' \in \partial \textbf{P}} P_{\gamma'}^{\# \{\text{arrows } [\gamma'] \to \gamma\}}}{\prod_{\gamma' \in \partial \textbf{P}} P_{\gamma'}^{\# \{\text{arrows } \gamma \to [\gamma']\}}}
\]

where $\# \{\text{arrows } [\gamma'] \to \gamma \}= \sum_{\tau \in [\gamma']} \# \{\text{arrows } \tau \to \gamma \text{ in } \overline{Q}(T)\}$ and $\# \{\text{arrows } \gamma \to [\gamma']\}=\sum_{\tau \in [\gamma']} \# \{\text{arrows } \gamma \to \tau \text{ in } \overline{Q}(T)\}$.

These seeds form $\mc{R}(Z_n)$, an $\A$-seed pattern of type $Z_n$ over $\QQ\PP$ where $\PP=$Trop$(P_\gamma: \gamma \subseteq \partial \textbf{P})$. The reader familiar with cluster algebras of geometric type should note that we could equally define this seed pattern by declaring $P_\gamma$ a frozen variable for $\gamma$ a boundary arc, defining an extended exchange matrix $\overline{B}(T)$ from the quiver $\overline{Q}(T)$, and then finding the $\X$-variable associated to each arc via the formula \cite[Equation 2.13]{FZ4}. 

In $\widehat{\mc{R}(Z_n)}$, the $\X$-variables of $\hat{\Sigma}_T$ are also indexed by arcs of $T$. To emphasize that the $\X$-variable associated to $\gamma$ in $\hat{\Sigma}_T$ depends only on the quadrilateral $q:=q_T(\gamma)$ of $\gamma$, we denote it by $\hat{x}_{q, \gamma}$. In simply laced types,

\begin{equation}\label{eqn:xhat}
\hat{x}_{q, \gamma}=\frac{\prod_{\gamma' \in T \cup \partial \textbf{P}} P_{\gamma'}^{\# \{\text{arrows } \gamma' \to \gamma\}}}{\prod_{\gamma' \in T \cup \partial \textbf{P}} P_{\gamma'}^{\# \{\text{arrows } \gamma \to \gamma'\}}}
\end{equation}

and in non-simply laced types,

\[\hat{x}_{q, [\gamma]}=\frac{\prod_{\gamma' \in T \cup \partial \textbf{P}} P_{\gamma'}^{\# \{\text{arrows } [\gamma'] \to \gamma\}}}{\prod_{\gamma' \in T \cup \partial \textbf{P}} P_{\gamma'}^{\# \{\text{arrows } \gamma \to [\gamma']\}}}.
\]

To show that $\hat{x}_{q, \gamma}$ and $\hat{x}_{q', \gamma'}$ are distinct for $q \neq q'$, we show that they evaluate differently at specific elements of $V$.

In the subsequent sections, we again label the $m$ vertices of $\textbf{P}$ clockwise with $1, \dots, m$.

\subsection{Type $A_n$}

Let $V=$Mat$_{2, n+3}(\CC)$ be the vector space of $2 \times (n+3)$ complex matrices. Let $\textbf{k}$ be the field of rational functions on $V$, written in terms of the coordinates of the column vectors $v_1, \dots, v_{n+3}$. Let the \emph{Pl\"{u}cker coordinate} $P_{ij} \in \textbf{k}$ be the determinant of the $2 \times 2$ matrix with columns $v_i$ and $v_j$. Pl\"{u}cker coordinates will be the basis of all of the following constructions. They give an embedding of the Grassmannian Gr$_{2, n+3}$ of $2$-planes in $\CC^{n+3}$ into complex projective space of dimension $\binom{n+3}{2}-1$, as the Pl\"{u}cker coordinates of a matrix $z \in V$ (up to simultaneous rescaling) depend only on the rowspan of $z$. The Pl\"{u}cker coordinates also generate the coordinate ring of the affine cone over the Grassmannian Gr$_{2, n+3}$ in the Pl\"{u}cker embedding.

If $\gamma$ is an arc or boundary segment of $\textbf{P}_{n+3}$ between vertices $i$ and $j$ ($i<j$), we define $P_\gamma:=P_{ij}$. In other words, the $\A$-variables of $\mc{R}(A_n)$ are exactly the Pl\"{u}cker coordinates $P_{ij}$ where $i<i+1<j$. The semifield $\PP$ is the tropical semifield generated by all consecutive Pl\"{u}cker coordinates $P_{i, i+1}$. For $\gamma$ an arc with quadrilateral $q$ in some triangulation, the $\X$-variable $x_{q, \gamma}$ records which boundary arcs $\tau$ are in triangles with $\gamma$. By the construction of $\overline{Q}(T)$ and \eqref{eqn:xvar}, if $\tau$ precedes $\gamma$ moving around the triangle clockwise, $P_\tau$ appears in the numerator of $x_{q, \gamma}$. If $\tau$ instead follows $\gamma$, $P_\tau$ appears in the denominator. 

\begin{figure}
\centering
\includegraphics[height=1.5in]{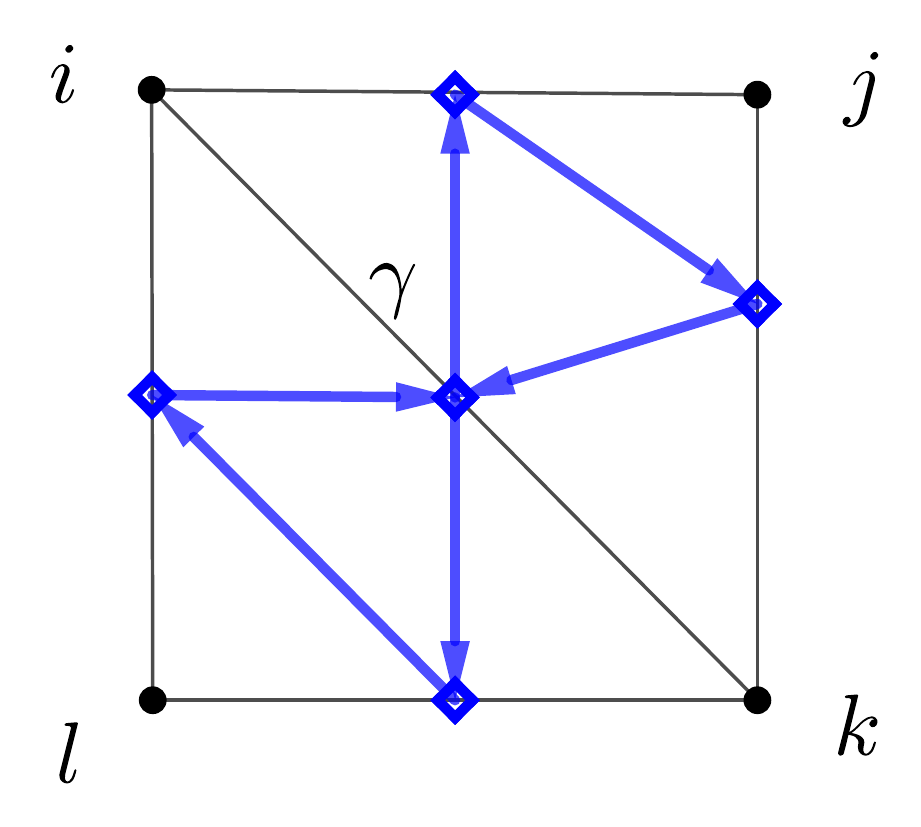}
\caption{\label{fig:Aquad} A quadrilateral and its associated quiver.}
\end{figure}

Let $T$ be a triangulation of $\textbf{P}_{n+3}$ containing an arc $\gamma$ with vertices $i$ and $k$. Suppose $q:=q_T(\gamma)$ has vertices $1\leq i<j<k<l \leq n+3$ (see Figure \ref{fig:Aquad}). Then \eqref{eqn:xhat} becomes

\begin{equation} \label{eqn:typeA} \hat{x}_{q, \gamma}=\frac{\prod\limits_{\tau: \tau \to \gamma}P_\tau}{\prod\limits_{\tau: \gamma \to \tau}P_\tau}=\frac{P_{il}P_{jk}}{P_{ij}P_{kl}}.
\end{equation}

The $\X$-variable associated to the other diagonal of the quadrilateral is $\hat{x}_{q, \gamma}^{-1}$. Clearly, $\hat{x}_{q,\gamma} \neq \hat{x}_{q, \gamma}^{-1}$, so the $\X$-variables associated to the two diagonals of the same quadrilateral are distinct. 

Consider another quadrilateral $q'$ with vertices $1 \leq a<b<c<d \leq n+3$ and diagonal $\gamma$. Choose $s \in \{a, b, c, d\}\setminus \{i, j, k, l\}$. The Pl\"{u}cker coordinate $\pm P_{st}$ appears in $\hat{x}_{q', \gamma}^{\pm 1}$ for some $t \in \{a, b, c, d\}$, Let $z$ be any matrix such that $\{v_i(z), v_j(z), v_k(z), v_l(z)\}=\{(1, 0),(0, 1),(1, 1),(-1, 1)\}$ and $v_s(z)=v_t(z)$. Then $\hat{x}_{q, \gamma}(z)$ is nonzero and $\hat{x}_{q', \gamma}^{\pm 1}(z)$ is either zero or undefined. This completes the proof of Proposition \ref{prop:injection} for type $A_n$.

\subsection{Type $B_n$} 

Let $V=$ Mat$_{2, n+2}(\CC)$ be the vector space of $2 \times (n+2)$ complex matrices. Again, let $\textbf{k}$ be the field of rational functions on $V$, written in terms of the coordinates of the column vectors $v_1, \dots, v_{n+3}$ and $\textbf{P}_{ij} \in \textbf{k}$ be the determinant of the $2 \times 2$ matrix with columns $v_i$ and $v_j$. We define a \emph{modified Pl\"{u}cker coordinate} $P_{i\overline{j}}:=P_{i, n+2}P_{j, n+2}-P_{ij}$.

Let $\gamma$ be an arc or boundary segment of $\textbf{P}_{n+1}^{\bullet}$. If $\gamma$ has endpoints $i, j \in \{1, \dots, n+1\}$, let $P_{[\gamma]}:=P_{ij}$ (respectively $P_{[\gamma]}:=P_{i \overline{j}}$) if it does not (respectively, does) cross the cut. If $\gamma$ is a radius with endpoints $i$ and $p$, then let $P_{[\gamma]}:=P_{i, n+2}$.

Let $q$ be a quadrilateral with vertices $1\leq i<j<k<l \leq n+1$. Then 
\begin{equation} \label{eqn:typeB4vert}
\hat{x}_{q, [\gamma]} \in \left \{ \left (\frac{P_{il}P_{jk}}{P_{ij}P_{kl}} \right)^{\pm 1}, \left (\frac{P_{i\overline{l}}P_{j\overline{k}}}{P_{ij}P_{kl}} \right)^{\pm 1}, \left (\frac{P_{i\overline{l}}P_{jk}}{P_{i\overline{j}}P_{kl}}\right)^{\pm 1}, \left (\frac{P_{i\overline{l}}P_{jk}}{P_{ij}P_{k\overline{l}}}\right)^{\pm 1}\right \}.
\end{equation}

Consider a quadrilateral $q$ with two vertices $i$ and $j$ ($i<j$) and let $\gamma$ be the plain radius with endpoints $i$ and $p$. Then 

\begin{equation} \label{eqn:typeB2vert}
\hat{x}_{q, [\gamma]}=\frac{P_{ij}^2}{P_{i\overline{j}}^2}.
\end{equation}

Finally, given a quadrilateral $q$ with three vertices $i<j< k$ and diagonal $\gamma$,

\begin{equation} \label{eqn:typeB3vert}
 \hat{x}_{q, [\gamma]} \in \left
\{ \left (\frac{P_{ij}P_{k, n+2}^2}{P_{i\overline{k}}P_{jk}} \right) ^{\pm 1}, \left ( \frac{P_{i\overline{k}}P_{j, n+2}^2}{P_{ij}P_{jk}} \right )^{\pm 1}, \left ( \frac{P_{jk}P_{i, n+2}^2}{P_{i\overline{k}}P_{ij}} \right ) ^{\pm 1} \right \}.
\end{equation}

We would like to show that all of these expressions are distinct. Note that in general, expressions from quadrilaterals on different vertices are definitely distinct (as long as none are identically zero on their domains, which follows easily from arguments below). Indeed, if $a$ is a vertex of $q'$ and not $q$, then one can find a matrix $z$ such that $\hat{x}_{q, \gamma}(z) \neq 0$ and this will not depend on $v_a(z)$. However, the index $a$ appears in the expression $\hat{x}_{q', \gamma'}$, and thus one can freely choose $v_a(z)$ in order to make $\hat{x}_{q', \gamma'}(z)$ either zero or undefined. 

The only instance in which this is not clear is when $\displaystyle \hat{x}_{q, \gamma}^{\pm1}= \frac{P_{i\overline{l}}P_{jk}}{P_{i\overline{j}}P_{kl}}$ and $q'$ is a quadrilateral with vertices $j$, $k$, and $l$. Because $i$ appears only in modified Pl\"{u}cker coordinates, it is not immediate that one can choose columns to make the expressions differ. The following remark gives a way to do this.

\begin{rmk}\label{rmk:typeB} $P_{a\overline{b}}=0$ for $v_a=(1, 0), v_b=(0, 1), v_{n+2}=(-1, 1)$, and, if $v_c=(1, 1), v_d=(-1, 1)$, $P_{s, \overline{t}}\neq 0$ for all other pairs $s, t \in \{a, b, c, d\}$ with $s \neq t$.
\end{rmk}

So there exists a matrix $z$ such that $P_{i \overline{k}}(z)=0$ and all other (usual or modified) Pl\"{u}cker coordinates involving the indices $i, j, k, l$ are nonzero, which covers the problematic case.

Remark \ref{rmk:typeB} also gives us that two $\X$-variables associated to different quadrilaterals on the same 4 vertices are distinct, as each expression in \eqref{eqn:typeB4vert} involves a different modified Pl\"{u}cker coordinate.

In the case when the quadrilaterals are on the same 3 vertices, if $\{v_i, v_j, v_k\}=\{(1, 0), (1, 1), (-1, 1)\}$, choosing $v_{n+2}=v_a$ for $a=i, j, k$ makes a unique expression in \eqref{eqn:typeB3vert} zero or undefined.

There are no two quadrilaterals on the same two vertices, and the expressions in \eqref{eqn:typeB2vert} are clearly not identically zero, so no further argument is needed.

\subsection{Type $C_n$} Let $V=$Mat$_{2, n+1}(\CC)$. Let $M \in$SO$_2(\CC)$ be
\[\begin{bmatrix}
0 & -1\\
1 & 0
\end{bmatrix}.
\]

Again, let $\textbf{k}$ be the field of rational functions on $V$, written in terms of the coordinates of the column vectors $v_1, \dots, v_{n+1}$ and $\textbf{P}_{ij} \in \textbf{k}$ be the determinant of the $2 \times 2$ matrix with columns $v_i$ and $v_j$. We define the \emph{modified Pl\"{u}cker coordinate} $P_{i \overline{j}}$ as the determinant of the $2 \times 2$ matrix whose first column is $v_j$ and whose second column is $Mv_i$. More explicitly, if $v_i=(v_{i, 1}, v_{i, 2})$, then $P_{i \overline{j}}=v_{j, 1}v_{i, 1}+v_{j, 2}v_{i, 2}$.

Let $\gamma$ be an arc or boundary segment of $\textbf{P}_{2n+2}$. We define $P_{[\gamma]}:=P_{ij}$ if $[\gamma]$ contains an arc from $i$ to $j$ with $i<j\leq n+1$, i.e. if one element of $[\gamma]$ is contained in a $\delta_1$-half-disk (see
 Definition \ref{defn:halfdisk}). We define $P_{[\gamma]}:=P_{i\overline{j}}$ if $[\gamma]$ contains an arc from $i$ to $j+n+1$ with $i \leq j \leq n+1$, i.e. if one element of $[\gamma]$ is not contained in a $\delta_1$-half-disk.

In a departure from earlier notation, we will use $i$ solely to denote a vertex in $\{1, \dots, n+1\}$; $i':=i+n+1$, as before. We will also include the vertices $1$ (respectively $n+2$)  in the $\delta_1$-half-disk containing $2$ (respectively, $n+3$). Recall that in $\textbf{P}_{2n+2}^G$, $q_T([\gamma])=[q_T(\gamma)]$ for some arc $\gamma$ and triangulation $T$ containing $\gamma$. If $[q_T(\gamma)]$ contains 2 quadrilaterals of $\textbf{P}_{2n+2}$, we will choose as a representative the quadrilateral containing the smallest vertex; if this does not determine one of the quadrilaterals, we will choose the quadrilateral containing the two smallest vertices. Quadrilaterals are given as tuples of their vertices, listed in clockwise order starting with the smallest vertex.

\begin{figure}
\centering
\includegraphics[width=0.8\textwidth]{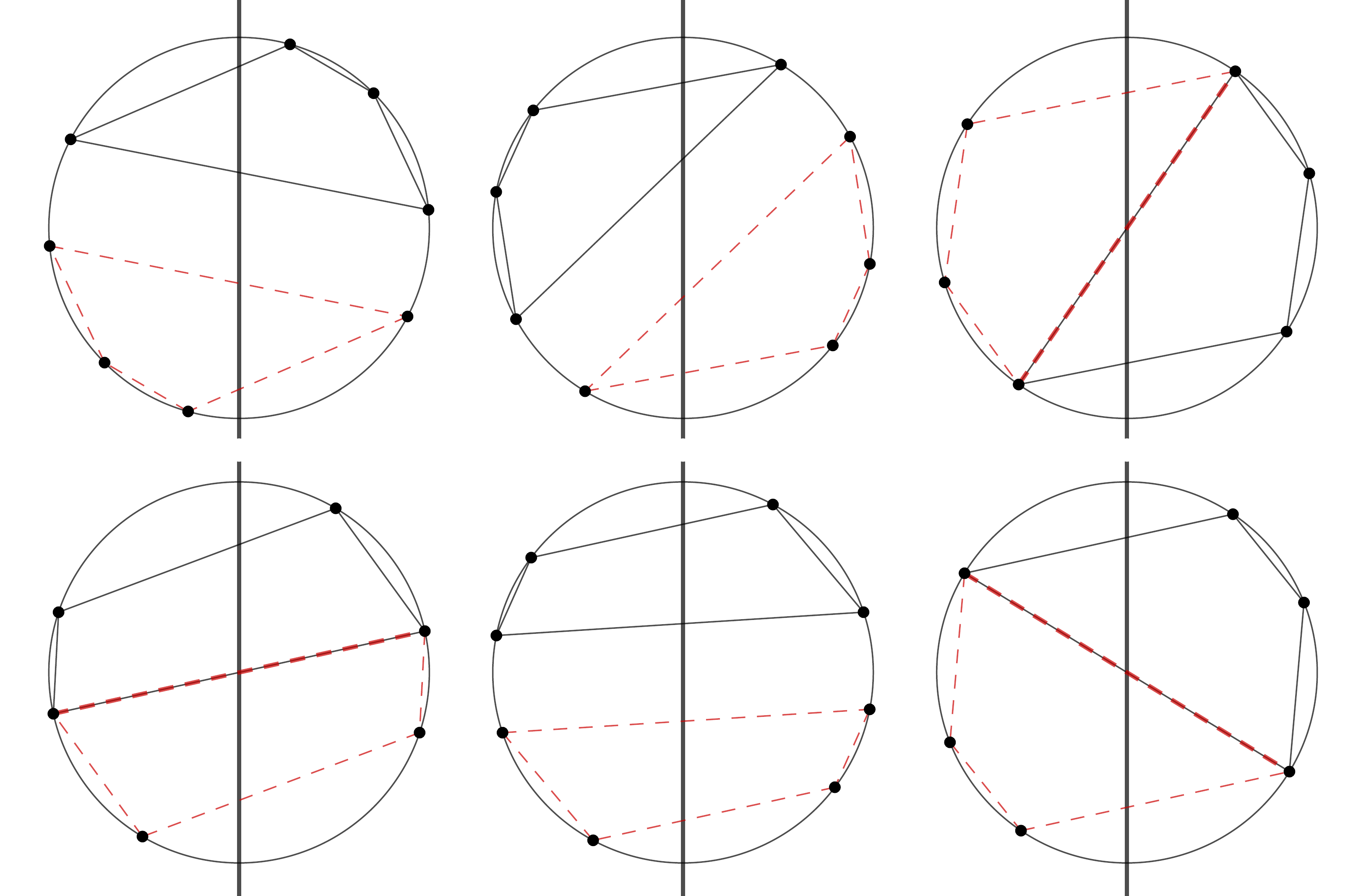}
\caption{\label{fig:typeCquad} The 6 types of quadrilaterals of $\textbf{P}_{2n+2}^{~G}$ that are not contained in $\delta_1$-half-disks but are contained in some closed half-disk. The thick vertical line is $\delta_1$. The distinguished representative of each quadrilateral is solid; the other representative is dashed. }
\end{figure}

If $q=[q_T(\gamma)]$ consists of a single quadrilateral, then $q_T(\gamma)=(i, j ,i', j')$. Then
\begin{equation*}
\hat{x}_{q, [\gamma]} \in \left \{ \left(\frac{P_{ij}^2}{P_{i \overline{j}}^2}\right)^{\pm 1} \right \}.
\end{equation*}

For the remaining cases, $q=[q_T(\gamma)]$ contains 2 quadrilaterals. If $q_T(\gamma)=(i, j, k, l)$ (i.e. all vertices are in one $\delta_1$-half-disk), then

\begin{equation*}
\hat{x}_{q, [\gamma]} \in \left \{ \left(\frac{P_{il}P_{jk}}{P_{ij}P_{kl}}\right)^{\pm 1} \right \}.
\end{equation*}

If $q_T(\gamma)$ has 3 vertices in one $\delta_1$-half-disk, $q_T(\gamma)$ is $(i, j, k, l')$ for $k<l$, $(i, j', k', l')$ for $i<j$, $(i, j, k, i')$, or $(i, j, k, k')$ (see Figure \ref{fig:typeCquad}). Then $\hat{x}_{q, [\gamma]}$ is, respectively

\begin{equation*}
\left (\frac{P_{i\overline{l}}P_{jk}}{P_{ij}P_{k\overline{l}}}\right )^{\pm 1},
 \left ( \frac{P_{i\overline{l}}P_{jk}}{P_{i\overline{j}}P_{kl}}\right )^{\pm 1}, \left (\frac{P_{i\overline{i}}P_{jk}}{P_{ij}P_{i \overline{k}}}\right )^{\pm 1}, \left (\frac{P_{i\overline{k}}P_{jk}}{P_{ij}P_{k\overline{k}}}\right )^{\pm 1}.
\end{equation*}

If $q_T(\gamma)$ has 2 vertices in each $\delta_1$-half-disk, $q_T(\gamma)$ is $(i,j, k', l')$ for $j<k$, or $(i, j, j', k')$ (see Figure \ref{fig:typeCquad}). Then $\hat{x}_{q, [\gamma]}$ is, respectively

\begin{equation*}
\left (\frac{P_{i\overline{l}}P_{j\overline{k}}}{P_{ij}P_{kl}}\right )^{\pm 1} \text{ or }\left (\frac{P_{i\overline{k}}P_{j\overline{j}}}{P_{ij}P_{jk}}\right )^{\pm 1}.
\end{equation*}

As long as none are identically zero, $\X$-variables from quadrilaterals involving different indices are distinct by a similar argument as the type $B_n$ case. Note that $M$ is invertible, so modified Pl\"{u}cker coordinates $P_{i \overline{j}}$ can be made zero by choosing either $c_i=M^{-1}c_j$ or $c_j=Mc_i$.

Consider a matrix with $v_a=(1, 0), v_b=(0, 1), v_c=(1, 1), v_d=(-2, 1)$. The only (usual or modified) Pl\"{ucker} coordinates that are zero are $P_{a\overline{b}}$ and $P_{b\overline{a}}$. Note also that $P_{a \overline{a}}>0$. This shows that no $\X$-variable is identically zero, and also that $\X$-variables involving the 4 same indices, but corresponding to different quadrilaterals, are distinct. 

To see that $\X$-variables containing the same 3 indices $i, j, k$ and corresponding to different quadrilaterals are distinct, note that they each have different values under the specialization $v_i=(1, 1), v_j=(-2, 1), v_k=(0, 1)$.

\subsection{Type $D_n$}

Let $V=$Mat$_{2, n}(\CC)$ and 

\[A:=\begin{bmatrix}
1 & 0\\
-1 & 2
\end{bmatrix}.
\] 

The eigenvalues of $A$ are $\lambda=1$ and $\overline{\lambda}=-1$; $a=(1 , 1)$ and $a^{\bowtie}=(0, -1)$ are eigenvectors for $\lambda$ and $\overline{\lambda}$ respectively. For $i<j$, we define a modified Pl\"{u}cker coordinate $P_{i \overline{j}}$ as the determinant of the $2 \times 2$ matrix with columns $v_j$ and $Av_i$. We also use the shorthand $P_{ia}$ (respectively $P_{i a^{\bowtie}}$) for the determinant of the $2 \times 2$ matrix with columns $v_i$ and $a$ (respectively $a^{\bowtie}$). 

Let $\gamma$ be an arc or boundary segment of $P_n^\bullet$. Then 

\[P_\gamma=
\begin{cases}
P_{ij} & \text{ if } \gamma \text{ has endpoints } $i, j$ \text{ with } i<j \text{ and does not cross the cut.}\\
P_{i\overline{j}} & \text{ if } \gamma \text{ has endpoints } $i, j$ \text{ with } i<j \text{ and crosses the cut.}\\
P_{ia} & \text{ if } \gamma \text{ is a plain radius with endpoints } p \text{ and } i.\\
P_{ia^{\bowtie}} & \text{ if } \gamma \text{ is a notched radius with endpoints } p \text{ and } i.\\
\end{cases}
\]

We make a slight modification to the usual recipe for producing an $\A$-seed pattern from this information. We add two additional vertices to $\overline{Q}(T)$, one labeled with $\lambda$ and the other with $\overline{\lambda}$. Let $T$ be the triangulation of $\textbf{P}_n^\bullet$ consisting only of radii. In $\overline{Q}(T)$, we add an arrow from $\lambda$ to the radius with endpoints $1, p$, and an arrow from the radius with endpoints $n, p$ to $\overline{\lambda}$. This is enough to determine the arrows involving $\lambda$ and $\overline{\lambda}$ for the remaining triangulations (for example, by performing $\X$-seed mutation on $\hat{\Sigma}_T$). 

Let $i$ and $j>i$ be vertices of $\textbf{P}_n^\bullet$, $\gamma$ be the plain radius from $j$ to $p$ and $\gamma'$ be the notched version of the same radius. Let $q$ be the quadrilateral on $i,j$ with $\gamma$ as a diagonal, and $q'$ be the quadrilateral on $i, j$ with $\gamma'$ as a diagonal. Then, by \cite[Proposition 5.4.11]{book}, $\hat{x}_{q, \gamma}$ and $\hat{x}_{q', \gamma'}$ are, respectively, 
\begin{equation} \label{eqn:typeD2vert}
\frac{\overline{\lambda}P_{ij}}{P_{\overline{ij}}}, \frac{\lambda P_{ij}}{P_{\overline{ij}}}.
\end{equation}
Their inverses are the $\X$-variables corresponding to plain and notched diagonals from $i$ to $p$, respectively.

For all other quadrilaterals, it suffices to consider the associated $\X$-variables up to some Laurent monomial in $\lambda, \overline{\lambda}$. Since $A$ is full rank, this ignored Laurent monomial will not impact the $\X$-variables being well-defined or nonzero, and we omit it for the sake of brevity.

If $q$ has vertices $i<j<k<l$, it does not enclose the puncture. Let $\gamma$ be a diagonal of $q$. Then
\begin{equation}\label{eqn:typeD4vert}
\hat{x}_{q, \gamma} \in 
\left \{ \left( \frac{P_{il}P_{jk}}{P_{ij}P_{kl}}\right )^{\pm 1}, \left( \frac{P_{i\overline{l}}P_{j\overline{k}}}{P_{ij}P_{kl}}\right )^{\pm 1}, \left( \frac{P_{i\overline{l}}P_{jk}}{P_{i\overline{j}}P_{kl}}\right )^{\pm 1}, \left( \frac{P_{i\overline{l}}P_{jk}}{P_{ij}P_{k\overline{l}}} \right )^{\pm 1}\right \}.
\end{equation}

If $q$ has vertices $i<j<k$ and has one diagonal $\gamma$ that is a radius, $\hat{x}_{q, \gamma}$ is one of

\begin{equation}	
\left (\frac{P_{ij}P_{ka}}{P_{ia}P_{jk}} \right )^{\pm 1}, \left ( \frac{P_{jk}P_{ia}}{P_{i\overline{k}}P_{ja}} \right )^{\pm 1}, \left (\frac{P_{i\overline{k}}P_{ja}}{P_{ij}P_{ka}}\right )^{\pm 1}
\end{equation}
or one of these expressions with $a^{\bowtie}$ instead of $a$.

If $q$ has vertices $i<j<k$ and both diagonals are arcs between vertices, then

\begin{equation}
\hat{x}_{q, \gamma} \in \left \{ \left(
\frac{P_{ij}P_{ka}P_{ka^{\bowtie}}}{P_{i\overline{k}}P_{jk}}\right )^{\pm 1}, \left ( \frac{P_{i\overline{k}}P_{ja}P_{ja^{\bowtie}}}{P_{ij}P_{jk}}\right )^{\pm 1}, \left (\frac{P_{jk}P_{ia}P_{ia^{\bowtie}}}{P_{i\overline{k}}P_{ij}}\right )^{\pm 1} \right \}.
\end{equation}

As with the other types, $\X$-variables of quadrilaterals involving different vertices are distinct, as long as none are identically zero. The matrix $A$ is of course invertible, so modified Pl\"{u}cker coordinates can be made zero regardless of which index is constrained. The rest of the argument is the same as with the other types. 

To see that $\X$-variables from quadrilaterals on the same 4 vertices are distinct, notice that if $v_a=(1, 0)$, $v_b=(-1, 1)$, and $\{v_c, v_d\}=\{(0, 1),(1, 1)\}$, then the only (usual or modified) Pl\"{u}cker coordinate that is zero is $P_{a\overline{b}}$. Since each expression in \eqref{eqn:typeD4vert} contains a unique modified Pl\"{u}cker coordinate, the result follows. (It is also clear that no expression is identically zero.)

For $\X$-variables from quadrilaterals on the same 3 vertices, it suffices to consider $\{v_i, v_j, v_k\}=\{(1, 0), (0, 1), (1, 1)\}$. Then $P_{i\overline{k}}\neq 0$. By making different columns equal to $a$ or $a^{\bowtie}$, one can differentiate between any 2 expressions. (Also, under the specialization $v_i=(0, 1)$, $v_j=(-1, 1)$, $v_k=(2, 1)$, all expressions are nonzero.)

For $\X$-variables from quadrilaterals on the same 2 vertices, the presence of $\lambda$ and $\overline{\lambda}$ serve to distinguish between them. For example, under the specialization $v_i=(1, 1)$, $v_j=(1, 0)$, the $\X$-variables in \eqref{eqn:typeD2vert} have different (nonzero) values. 

\section{Corollaries and Conjectures}

We make a few remarks regarding the implications of these results to the larger theory of $\X$-seed patterns and cluster algebras, and conjectural extensions.

First, we conjecture that Theorem \ref{quadbijection} extends to seed patterns from arbitrary marked surfaces.

 \begin{conj} Let $\mc{S}$ be an $\X$-seed pattern from a marked surface $(S, M)$. Then the map from $\{q_T(\gamma) \cup \{\gamma\}| ~ T \text{ a triangulation of } (S, M), \gamma \in T\}$ to $\X(\mc{S}_{sf})$, which sends $q \cup \{\gamma\}$ to $x_{q, \gamma}$, is a bijection.
 \end{conj}
 
  As remarked upon previously, a surjection from quadrilaterals (with choice of diagonal) to $\X$-variables holds by results of \cite{FST}.

Theorem \ref{quadbijection} implies that, for $\mc{S}_{sf}$ of classical type, the $\X$-variables of an $\X$-seed determine the $\X$-seed. Indeed, the $\X$-variables give the quadrilaterals of a tagged triangulation and a tagged triangulations is uniquely determined by its set of quadrilaterals. Thus each $\X$-seed corresponds to a unique triangulation, and we have the following corollary.

\begin{cor}
The exchange graph of the $\X$-seed pattern $\mc{S}_{sf}$ in classical types coincides with the exchange graph of any $\A$-seed pattern of the same type.
\end{cor}

Recall that the diagonals of a quadrilateral can be uniquely associated to a pair of $\A$-variables. These pairs are precisely those variables that appear together on the left hand side of an exchange relation; such $\A$-variables are called \emph{exchangeable}. Clearly, there is a bijection from ordered pairs of exchangeable $\A$-variables to quadrilaterals with a choice of diagonal. Composing this bijection with the bijection of Theorem \ref{quadbijection} gives Corollary \ref{expairbijection} for classical types, which we give again here for the reader's convenience. It was checked by computer for exceptional types.

\begin{manualcor}{1.2} Let $\mc{R}$ be a finite type $\A$-seed pattern. There is a bijection between ordered pairs of exchangeable $\A$-variables in $\mc{R}$ and $\X(\mc{S}_{sf})$.
\end{manualcor}

Let $\mc{R}$ is a finite type $\A$-seed pattern over the tropical semifield with one (equivalently, every \cite[Lemma 1.2]{GSV}) extended exchange matrix of full rank. As is remarked in \cite[Section 7]{FZ4}, in this case the $\X$-variables in the corresponding seed of $\hat{\mc{R}}$ are algebraically independent. Thus, the number of $\X$-variables in $\hat{\mc{R}}$ is $|\X(\mc{S}_{sf})|$. (Note that without this condition, the number of $\X$-variables in $\hat{\mc{R}}$ could be smaller, as Example \ref{ex:degenerateseed} shows.) Recall that in this setting, the $\X$-variables of $\hat{\mc{R}}$ exactly record the two terms on the right hand side of an exchange relation. In the bijection of Corollary \ref{expairbijection}, the pairs of exchangeable $\A$-variables are mapped to the $\X$-variables recording the exchange relation that the $\A$-variables satisfy. This implies the following corollary.

 \begin{cor}\label{uniqueexchange} Let $\mc{R}$ be an $\A$-seed pattern of classical type over the tropical semifield such that one (equivalently, every) extended exchange matrix is full rank. Then the two monomials on the right hand side of an exchange relation \eqref{exchangerel} uniquely determine the variables being exchanged.
 \end{cor}


Lastly, in the original development of finite type $\A$-seed patterns, seed patterns were connected to root systems of the same type. In particular, there is a bijection between $\A$-variables and almost positive roots (positive roots and negative simple roots). Two variables are exchangeable if and only if the corresponding roots $\alpha, \beta$ have $(\alpha|| \beta)=(\beta|| \alpha)=1$, where $(-||-)$ is the \emph{compatibility degree} \cite{FZ2}. This, combined with Corollary \ref{expairbijection}, give a root theoretic interpretation of the number of $\X$-coordinates of $\mc{S}_{sf}$.

\begin{cor} \label{cor:rootinterp}For $\mc{S}_{sf}$ of Dynkin type, $|\X(\mc{S}_{sf})|$ is the number of pairs of almost-positive roots $(\alpha, \beta)$ such that $(\alpha|| \beta)=(\beta|| \alpha)=1$ in the root system of the same type.
\end{cor}

\section{Acknowledgments} The author would like to thank Lauren Williams for helpful conversations and her comments on drafts of this paper, and Dan Parker and Adam Scherlis for making their undergraduate theses available. The author would also like to thank Nathan Reading for computing the number of pairs of exchangeable $\A$-variables for seed patterns of exceptional types, thereby extending Corollaries \ref{expairbijection} and \ref{cor:rootinterp} to exceptional types. This research did not receive any specific grant from funding agencies in the public, commercial, or not-for-profit sectors.

An extended abstract outlining these results was accepted for presentation at FPSAC 2018.

\printbibliography

\end{document}